\documentclass[11pt,reqno,tbtags,a4paper]{amsart}
\usepackage{amssymb}
\usepackage{mathabx}  
\usepackage{xpunctuate}
\usepackage{url}
\usepackage[square,numbers]{natbib}
\bibpunct[, ]{[}{]}{;}{n}{,}{,}

\title
{Cyclic and alternating $U$-statistics}

\date{14 October, 2025}

\author{Svante Janson}
\thanks{Supported by the Knut and Alice Wallenberg Foundation
and
the Swedish Research Council
}
\address{Department of Mathematics, Uppsala University, PO Box 480,
SE-751~06 Uppsala, Sweden}
\email{svante.janson@math.uu.se}
\newcommand\urladdrx[1]{{\urladdr{\def~{{\tiny$\sim$}}#1}}}
\urladdrx{http://www2.math.uu.se/~svante/}

\subjclass[2020]{} 

\numberwithin{equation}{section}

\renewcommand\le{\leqslant}
\renewcommand\ge{\geqslant}

\allowdisplaybreaks

\theoremstyle{plain}
\newtheorem{theorem}{Theorem}[section]
\newtheorem{lemma}[theorem]{Lemma}
\newtheorem{proposition}[theorem]{Proposition}
\newtheorem{corollary}[theorem]{Corollary}

\theoremstyle{definition}

\newcommand\xqed[1]{%
    \leavevmode\unskip\penalty9999 \hbox{}\nobreak\hfill
    \quad\hbox{#1}}

\newtheorem{exampleqqq}[theorem]{Example}
\newenvironment{example}{\begin{exampleqqq}}
  {\xqed{$\triangle$}\end{exampleqqq}}

\newtheorem{remarkqqq}[theorem]{Remark}
\newenvironment{remark}{\begin{remarkqqq}}
  {\xqed{$\triangle$}\end{remarkqqq}}

\newtheorem{problem}[theorem]{Problem}

\theoremstyle{remark}

\newcounter{dummy}
\makeatletter
\newcommand\myitem[1][]{\item[#1]\refstepcounter{dummy}\def\@currentlabel{#1}}
\makeatother

\newenvironment{romenumerate}[1][-10pt]{
\addtolength{\leftmargini}{#1}\begin{enumerate}
 \renewcommand{\labelenumi}{\textup{(\roman{enumi})}}%
 }{\end{enumerate}}

\newcounter{oldenumi}
{\setcounter{oldenumi}{\value{enumi}}
\begin{romenumerate} \setcounter{enumi}{\value{oldenumi}}}
{\end{romenumerate}}

\newcounter{thmenumerate}

\newcounter{xenumerate}   

\newcommand\pfitemx[1]{\par#1:}
\newcommand\pfitemref[1]{\pfitemx{\ref{#1}}}

\newcommand{\refT}[1]{Theorem~\ref{#1}}
\newcommand{\refTs}[1]{Theorems~\ref{#1}}
\newcommand{\refC}[1]{Corollary~\ref{#1}}
\newcommand{\refCs}[1]{Corollaries~\ref{#1}}
\newcommand{\refL}[1]{Lemma~\ref{#1}}
\newcommand{\refLs}[1]{Lemmas~\ref{#1}}
\newcommand{\refR}[1]{Remark~\ref{#1}}
\newcommand{\refRs}[1]{Remarks~\ref{#1}}
\newcommand{\refS}[1]{Section~\ref{#1}}
\newcommand{\refSs}[1]{Sections~\ref{#1}}
\newcommand{\refSS}[1]{Section~\ref{#1}}
\newcommand{\refP}[1]{Proposition~\ref{#1}}

\newcommand{\refE}[1]{Example~\ref{#1}}

\newcommand{\refApp}[1]{Appendix~\ref{#1}}

\newcommand\marginal[1]{\marginpar[\raggedleft\tiny #1]{\raggedright\tiny#1}}
\newcommand\kolla{\marginal{KOLLA!} }

\newcommand\REM[1]{{\raggedright\texttt{[#1]}\par\marginal{XXX}}}
\newcommand\XREM[1]{\relax}

\begingroup
  \divide\count255 by 60
  \multiply\count255 by -60
  \advance\count255 by \time
  \ifnum \count255 < 10 \xdef\klockan{\the\count1.0\the\count255}
  \else\xdef\klockan{\the\count1.\the\count255}\fi
\endgroup

\newcommand{\sumk}{\sum_{k=1}^\infty}
\newcommand{\sumkk}{\sum_{k=-\infty}^\infty}

\newcommand{\sumin}{\sum_{i=1}^n}
\newcommand{\sumjn}{\sum_{j=1}^n}

\newcommand{\sumrR}{\sum_{r=1}^R}

\newcommand{\sumqQp}{\sum_{q=1}^\Qp}

\newcommand{\prodk}{\prod_{k=1}^\infty}

\newcommand\set[1]{\ensuremath{\{#1\}}}

\newcommand\Bigset[1]{\ensuremath{\Bigl\{#1\Bigr\}}}

\newcommand\xpar[1]{(#1)}
\newcommand\bigpar[1]{\bigl(#1\bigr)}
\newcommand\Bigpar[1]{\Bigl(#1\Bigr)}
\newcommand\biggpar[1]{\biggl(#1\biggr)}
\newcommand\lrpar[1]{\left(#1\right)}
\newcommand\bigsqpar[1]{\bigl[#1\bigr]}
\newcommand\sqpar[1]{[#1]}
\newcommand\Bigsqpar[1]{\Bigl[#1\Bigr]}

\newcommand\cpar[1]{\{#1\}}

\newcommand\Bigabs[1]{\Bigl\lvert#1\Bigr\rvert}

\newcommand\lrabs[1]{\left\lvert#1\right\rvert}
\def\rompar(#1){\textup(#1\textup)}    

\newcommand\Bigparfrac[2]{\Bigpar{\frac{#1}{#2}}}

\newcommand\innprod[1]{\langle#1\rangle}

\def\xexp(#1){e^{#1}}

\newcommand\floor[1]{\lfloor#1\rfloor}
\newcommand\lrfloor[1]{\left\lfloor#1\right\rfloor}

\newcommand\setn{\set{1,\dots,n}}

\newcommand\ntoo{\ensuremath{{n\to\infty}}}
\newcommand\Ntoo{\ensuremath{{N\to\infty}}}

\newcommand\norm[1]{\lVert#1\rVert}

\newcommand\punkt{\xperiod}    
\newcommand\iid{i.i.d\punkt}    

\newcommand\eg{e.g\punkt}

\newcommand\cf{cf\punkt}

\newcommand\ii{\mathrm{i}}

\newcommand{\tend}{\longrightarrow}
\newcommand\dto{\overset{\mathrm{d}}{\tend}}
\newcommand\pto{\overset{\mathrm{p}}{\tend}}
\newcommand\asto{\overset{\mathrm{a.s.}}{\tend}}

\newcommand\eqd{\overset{\mathrm{d}}{=}}

\newcommand\op{o_{\mathrm p}}

\newcommand\bbR{\mathbb R}
\newcommand\bbC{\mathbb C}
\newcommand\bbN{\mathbb N}

\newcommand\bbZ{\mathbb Z}

\newcounter{CC}
\newcounter{cc}

\newcommand\E{\operatorname{\mathbb E}{}} 

\newcommand\Var{\operatorname{Var}}
\newcommand\Cov{\operatorname{Cov}}

\newcommand\Be{\operatorname{Be}}

\newcommand\sgn{\operatorname{sgn}}

\newcommand\ga{\alpha}
\newcommand\gb{\beta}
\newcommand\gd{\delta}

\newcommand\gf{\varphi}
\newcommand\gF{\gf}

\newcommand\gG{\Gamma}
\newcommand\gk{\varkappa}

\newcommand\gl{\lambda}

\newcommand\go{\omega}

\newcommand\gs{\sigma}

\newcommand\gss{\sigma^2}

\renewcommand\phi{\xxx}  

\newcommand\cA{\mathcal A}
\newcommand\cB{\mathcal B}

\newcommand\cE{\mathcal E}

\newcommand\cN{\mathcal N}

\newcommand\cQ{\mathcal Q}
\newcommand\cR{{\mathcal R}}
\newcommand\cS{{\mathcal S}}

\newcommand\cX{{\mathcal X}}
\newcommand\cY{{\mathcal Y}}
\newcommand\cZ{{\mathcal Z}}

\newcommand\tU{{\tilde U}}

\newcommand\tX{{\widetilde X}}

\newcommand\indic[1]{\boldsymbol1\cpar{#1}}

\newcommand\etta{\boldsymbol1} 

\newcommand\smatrixx[1]{\left(\begin{smallmatrix}#1\end{smallmatrix}\right)}
\newcommand\matrixx[1]{\begin{pmatrix}#1\end{pmatrix}}

\newcommand\qw{^{-1}}

\newcommand\qq{^{1/2}}
\newcommand\qqw{^{-1/2}}

\newcommand\intoi{\int_0^1}

\newcommand\oi{\ensuremath{[0,1]}}

\newcommand\ooo{[0,\infty)}

\newcommand\dd{\,\mathrm{d}}

\newcommand{\chf}{characteristic function}

\newcommand\lhs{left-hand side}
\newcommand\rhs{right-hand side}

\newcommand\xoo{_1^\infty}
\newcommand\xoon{_1^n}

\newcommand\fS{\mathfrak{S}}

\newcommand\Uc{U^\circ}
\newcommand\Ucc{\widetilde U^\circ}
\newcommand\Ustat{$U$-statistic}

\newcommand\Upm{U^{+-}}
\newcommand\Ump{U^{-+}}
\newcommand\Umm{U^{--}}
\newcommand\xf{f^*}
\newcommand\hxf{{\widehat{\xf}}}
\newcommand\UU{\bar{U}}
\newcommand\SO{S^{(\emptyset)}}
\newcommand\SI{S^{(1)}}
\newcommand\SII{S^{(2)}}
\newcommand\SQ{S^{(12)}}
\newcommand\LLR{L^2_{\bbR}}
\newcommand\LLC{L^2_{\bbC}}
\newcommand\XX{X}
\newcommand\hcX{\widehat{\cX}}
\newcommand\hx{\widehat x}
\newcommand\hy{\widehat y}
\newcommand\hX{\widehat X}
\newcommand\hnu{\widehat\nu}
\newcommand\hf{{\widehat{f}}}
\newcommand\hF{{\widehat{F}}}
\newcommand\czf{{\widecheck{f}}}
\newcommand\Leb{\ell}
\newcommand\fmu{{f-\mu}}
\newcommand\hfmu{{\hf-\mu}}
\newcommand\OLL{O_{L^2}}
\newcommand\oLL{o_{L^2}}

\newcommand\flfrac[2]{\floor{\frac{#1}{#2}}}
\newcommand\lrflfrac[2]{\lrfloor{\frac{#1}{#2}}}
\newcommand\flfracnn{\flfrac{n}{2}}
\newcommand\lrflfracnn{\lrflfrac{n}{2}}
\newcommand\MM{M_1\oplus M_2}
\newcommand\fs{{f_{\mathsf s}}}
\newcommand\fa{{f_{\mathsf a}}}
\newcommand\fiiis{{f_{12\mathsf s}}}
\newcommand\fiiia{{f_{12\mathsf a}}}
\newcommand\gls{\gl^{\mathsf s}}
\newcommand\gla{\gl^{\mathsf a}}
\newcommand\gll{\gla}
\newcommand\cQp{{\cQ_+}}
\newcommand\Qp{{Q_+}}
\newcommand\hfiii{{\widehat{f_{12}}}}
\newcommand\czfiii{{\widecheck{f_{12}}}}
\newcommand\xfiii{{(f_{12})^*}}
\newcommand\gsse{\gss_{\mathsf e}}
\newcommand\gsso{\gss_{\mathsf o}}
\newcommand\tzeta{\tilde{\zeta}}
\newcommand\bgs{\boldsymbol{\gs}}
\newcommand\tensor{\otimes}
\newcommand\setiii{\set{1,2}}
\newcommand\xeta{\vartheta}
\newcommand\Hs{{H_{\mathsf s}}}
\newcommand\Ha{{H_{\mathsf a}}}
\newcommand\Ws{{W_{\mathsf s}}}
\newcommand\Wa{{W_{\mathsf a}}}
\newcommand\mpp{measure-preserving}
\newcommand\weta{\widetilde\eta}
\newcommand\wxeta{\widetilde\xeta}
\newcommand\goo{\go_0}
\newcommand\st{_{s,\tau}}
\newcommand\hfst{{\widehat{f\st}}}

\newcommand\CS{Cauchy--Schwarz}
\newcommand\CSineq{\CS{} inequality}
\newcommand{\Levy}{L\'evy}
\newcommand{\Cramer}{Cram\'er}

\begin{document}

\begin{abstract} 
We define \emph{cyclic $U$-statistics}
as a variant of $U$-statistics based on variables $X_1,\dots,X_n$
that are assumed to be cyclically ordered.
We also define \emph{alternating $U$-statistics} where in the definition
terms are summed with alternating sings (in three different ways).
Only $U$-statistics of order 2 are considered.
The definitions are inspired by special cases studied by 
Chebikin (2008) and Even-Zohar (2017) for random permutations.

We show limit theorems similar to well-known results for standard
$U$-statistics, but with some differences between the different versions.
In particular, we find both ``nondegenerate'' normal limits and
``degenerate'' non-normal limits.
\end{abstract}

\maketitle

\section{Introduction}\label{S:intro}

$U$-statistics were introduced by \citet{Hoeffding} as statistics of the
form
\begin{align}
  \label{U1}
\UU_n =\UU_n(f)=\UU_n(f;X_1,\dots,X_n):= \sum f(X_{i_1},\dots,X_{i_m}) 
\end{align}
where 
$m$ (the \emph{order} of the $U$-statistic) and $n$ are integers with
$1\le m\le n$,
the sum is over all $(n)_m=n!/(n-m)!$ different 
$m$-tuples $i_1,\dots,i_m$ of distinct
indices in $\setn$,
$X_1,\dots,X_n$ is a sequence of random variables taking values in
some measurable space $\cX$,
and
$f:\cX^m\to \bbR$ is a measurable function, called the \emph{kernel} of the
\Ustat. 
(\citet{Hoeffding} and many later authors
include in the definition 
a normalization factor $1/(n)_m$; this is often convenient, but in the
present paper we choose to omit such factors in the definitions.)
The random variables $X_i$ are usually assumed to be independent and
identically distributed (\iid), and this will  be assumed in the
present paper.
We will also assume that the kernel $f$ is square integrable
in the sense
$\E |f(X_1,\dots,X_m)|^2<\infty$,
which we write as
$f\in L^2=L^2(\cX^m)$.

In the definition \eqref{U1}, the order of the variables $X_1,\dots,X_n$
does not matter; in other words, the indices $1,\dots,n$ are used for
labelling but their order does not matter and any other labels could be used.
Another definition of $U$-statistics where the order of the variables matters
is obtained by summing only over increasing sequences $i_1<\dots<i_m$:
\begin{align}
  \label{U2}
U_n =U_n(f)=U_n(f;X_1,\dots,X_n):= \sum_{i_1<\dots<i_m} f(X_{i_1},\dots,X_{i_m}).
\end{align}
Note that if $f$ is symmetric, as is often assumed,
the definitions \eqref{U1} and \eqref{U2}
differ only by an unimportant factor $m!$. In fact, $\UU_n$ in
\eqref{U1} remains the same if $f$ is replaced by its symmetrization;
hence we may as well assume that $f$ in \eqref{U1} is symmetric, and
therefore \eqref{U1} can be seen as a special case of the more general
\eqref{U2}. (Although many applications use only the symmetric version
\eqref{U1}, or equivalently \eqref{U2} with a symmetric kernel $f$,
there are also many applications that require the asymmetric version
\eqref{U2}.)
The literature on \Ustat{s} and applications of them is enormous, and we
will in the sequel only give a few relevant references.

In the present paper we consider 
four modifications of the definitions above, 
defined in the following two subsections. 
As in \eqref{U1} and \eqref{U2}, we may include $f$ and the variables $X_i$
in the notation, but we often omit them when they are clear from the context.
Our main results are theorems on the asymptotic distribution of these
modifications, stated in \refSs{SUc}--\ref{SA} and summarized for
convenience in \refS{SSummary}.
The results are similar to the well-known results for classical \Ustat{s},
which we state for comparison in \refS{SU}, 
although the details in the limit theorems differ between
the different versions.
In particular, in the classical case $U_n(f)$ there is a well-known
dichotomy of the kernels $f$
into a nondegenerate case with variance of order $n^3$ and
an asymptotically normal distribution,
and a degenerate case with a variance of order $n^2$ only
and with a non-normal limit distribution; 
all but one of the versions studied here
exhibit the same two cases, but not necessarily for the same kernels;
however, for one version (\refT{TB}) there is no ``nondegenerate'' case.
Furthermore, for some versions there is an exceptional third, futher
degenerate and almost trivial case, with variance of order $n$ and again a
normal limit distribution.
In the degenerate case, for all versions, 
the limit distribution can be expressed as a
(possibly infinite) linear combination of centred squares of independent
standard normal variables, where the coefficients are the eigenvalues of a
certain integral operator with kernel derived from the kernel $f$ of the
\Ustat, although again the details differ between the differnt versions;
see \eg{} \eqref{tu2}.
(From an abstract point of view, the limit distribution is a Wiener chaos of
order 2, see \eg{} \cite[in particular Theorem 6.1]{SJIII}.)

Some simple examples are given in \refS{Sex}; this includes applications to
the \emph{writhe} and \emph{alternating inversion number} of a uniformly
random permutation, previously defined and studied in \cite{writhe} and
\cite{Chebikin}. These two examples were the inspiration of the definitions
and results in the present paper.

Some further results and open problems are given in \refS{Sfurther}.

\refApp{AA} contains the proof of the well-known \refT{TU} for
classical \Ustat{s}, included for completeness; we also reuse parts of the
appendix in other proofs.
\refApp{Acumulants} collects some formulas for cumulants.
\refApp{AYor} gives further calculations for one example from \refS{Sex}.

As mentioned above, in the degenerate cases, the asymptotic distribution is
described by the eigenvalues of an integral operators.
Consequently, some proofs and most examples require finding such eigenvalues; 
this is straightforward in our cases, but we include details
for completeness.

\subsection{Cyclic \Ustat{s}}
For this version, we
regard the indices $1,\dots,n$ as
circularly ordered instead of linearly ordered.
We regard the indices as elements of the cyclic group $\bbZ_n=\bbZ/n\bbZ$;
we therefore assume again that $X_1,\dots,X_n$ are \iid{} random variables, and
extend the notation to $X_i$ for all $i\in\bbZ$ by $X_i:=X_j$ if $i\equiv j
\pmod n$ and $j\in\setn$.

We consider only the case $m=2$ and then define the \emph{cyclic $U$-statistic}
\begin{align}\label{Uc}
  \Uc_n=\Uc_n(f):=\sum_{i\in\bbZ_n}\sum_{1\le j < n/2} f(X_i,X_{i+j}).
\end{align}
Note that if we regard the elements of $\bbZ_n$ as lying on a circle in the
natural way, then for any pair of distinct $i,j\in\bbZ_n$, the sum
\eqref{Uc} contains the term $f(X_i,X_j)$ if the shortest path from $i$ to $j$
goes in the positive direction, and the term $f(X_j,X_i)$ if the shortest path
goes in the negative direction; for even $n$, no terms $f(X_i,X_{i+n/2})$
appears at all.

One example of a cyclic \Ustat{} appears in a paper by \citet{writhe}
where he studies the \emph{writhe} of permutations and framed knots,
see \refE{Ewrithe} for details. 
In particular, \cite{writhe} studies the writhe of a uniformly random
permutation 
and finds its asymptotic distribution.
The writhe of a uniformly random permutation can be written as 
a cyclic \Ustat{} \eqref{Uc}
(see \refE{Ewrithe} again),
and this example is the motivation for the present paper, where 
we study general cyclic \Ustat{s} (assuming only that $f\in L^2$) 
and prove general limit theorems; in particular, this
gives an alternative proof of the limit theorem by \citet{writhe}
(where the theorem is proved by quite different methods).

\begin{remark}\label{Rsymm}
  If $f(x,y)$ is a symmetric function, then we have
  \begin{align}
    \Uc_n=
    \begin{cases}
      U_n, & n \text{ is odd}
\\
     U_n - \sum_{i=1}^{n/2}f(X_i,X_{i+n/2}) & n \text{ is even},
    \end{cases}
  \end{align}
where it is easily seen that the sum appearing for even $n$ is
asymptotically negligible.
(If $f\in L^2$, this sum has variance $O(n)$ since its terms are \iid, 
while $U_n$ except in trivial
cases has variance of order at least $n^2$.)
Hence, we do not really obtain anything new for symmetric $f$.
The main interest seems to be in the opposite case, when $f$ is
antisymmetric. We will in the sequel
give  special attention to the cases of symmetric and antisymmetric kernels.

Note that every $f$ may be decomposed as a sum of a symmetric and an
antisymmetric part; hence a $U$-statistic of order $m=2$
(of any of the versions studied here)
can be written as a sum of a symmetric and an antisymmetric $U$-statistic;
We will see in \refT{TC} and \refR{Rdecouple} that for the cyclic \Ustat{}
$\Uc_n$, the
two parts are asymptotically independent, and thus can be treated separately.
However, for the other \Ustat{s} considered here, including the classical
$U_n$,
this decomposition is of limited value since the two parts typically are
dependent, also asymptotically. 
For an interesting example of this, largely taken from \cite{SJ22}, see
\refE{ESJ22}.
\end{remark}

\begin{remark}\label{RUcc}
  A variation of the definition \eqref{Uc} is the more symmetrical
\begin{align}\label{Ucc}
  \Ucc_n(f):=\sum_{i\in\bbZ_n}\sum_{1\le j < n/2} f(X_i,X_{i+j})
-\sum_{i\in\bbZ_n}\sum_{1\le j < n/2} f(X_i,X_{i-j}).
\end{align}
However, replacing $i$ by $i+j$ in the second sum yields
\begin{align}\label{Ucc2}
  \Ucc_n(f)=\sum_{i\in\bbZ_n}\sum_{1\le j < n/2} 
\bigpar{f(X_i,X_{i+j})-f(X_{i+j},X_i)}=\Uc_n(g),
\end{align}
where $g(x,y):=f(x,y)-f(y,x)$. 
Hence, cyclic $U$-statistics of the form \eqref{Ucc}
are special cases of $\Uc_n$, and therefore need not be considered further.
Note that the function $g$ arising here always is antisymmetric;
conversely, if $f$ is antisymmetric, then \eqref{Ucc2} implies
$\Uc_n(f)=\Ucc_n(\frac12f)$.
Consequently, the version $\Ucc_n$ is 
equivalent to the special case of $\Uc_n$ for antisymmetric kernels only.
\end{remark}

\subsection{Alternating \Ustat{s}}
One of the results in \cite{writhe} is that the writhe of a uniformly random
permutation has the same distribution as the 
\emph{bi-alternating inversion number}, which is defined in \cite{writhe} 
in analogy to the \emph{alternating inversion number} defined in
\cite{Chebikin}, see \refE{Einv};
it is also shown in \cite{writhe} that the alternating and
bi-alternating inversion numbers of a uniformly random permutation,
although similarly defined,
have quite different asymptotic distributions. 

These numbers are defined as versions of the usual inversion number of a
permutation.
It is well-known that the usual inversion number of a uniformly random
permutation can be written as a \Ustat. Again we study 
corresponding modifications of general \Ustat{s}.
We consider again only the case of order $m=2$, so that \eqref{U2} becomes
\begin{align}\label{U3}
  U_n=U_n(f):=\sum_{1\le i<j\le n} f(X_i,X_j).
\end{align}
We then define the \emph{alternating \Ustat{s}}
\begin{align}
\Ump_n=\Ump_n(f)&:=\sum_{1\le i<j\le n} (-1)^{i+1}f(X_i,X_j),\label{U-+}
\\
  \Upm_n=\Upm_n(f)&:=\sum_{1\le i<j\le n} (-1)^{j}f(X_i,X_j),\label{U+-}
\\
  \Umm_n=\Umm_n(f)&:=\sum_{1\le i<j\le n} (-1)^{i+j}f(X_i,X_j).\label{U--}
\end{align}
We also call $\Umm_n$ \emph{bi-alternating}.
(As in  \cite{writhe} for a special case, see \refE{Ewrithe}.)

\begin{remark}
  Define
  \begin{align}\label{fi1}
    \xf(x,y):=f(y,x).
  \end{align}
Then, by replacing $X_i$ by $X_{n+1-i}$, it follows that
\begin{align}\label{fi2}
    \Upm_n(f)
&\eqd \sum_{1\le i<j\le n} (-1)^{j}f(X_{n+1-i},X_{n+1-j})
=
(-1)^{n}\sum_{1\le k<\ell\le n} (-1)^{k+1}f(X_{\ell},X_{k})
\notag\\&
=(-1)^{n}\Ump_n(\xf).
\end{align}
Consequently, up to a trivial change of sign and replacing $f$ by $\xf$,
$\Upm_n$ is the same as $\Ump_n$, and thus it suffices to consider the
latter.
On the other hand,
as noted in \cite{writhe} for the example in \refE{Einv}, 
we will see
that $\Umm$ in general is quite different.
\end{remark}

\section{Preliminaries}\label{Sprel}

\subsection{Some notation}

We assume throughout that $X_1,X_2,\dots$ are \iid{} random variables with
values in some measurable space $\cX$, 
and let $\nu$ be the common distribution of $X_i$.
We let $\XX$ denote any random variable with this distribution.
Thus $(\cX,\nu)$ is a probability space, which we for simplicity also
denote by $\cX$.

We define 
$\hcX:=\cX\times\oi$, where (as always below), $\oi$ is equipped with the
Lebesgue measure $\Leb$.

We assume also that $f:\cX^2\to\bbR$ is a given function such
that $f\in L^2(\cX^2)$, i.e., 
\begin{align}\label{a1}
\E|f(X_1,X_2)|^2=\int_{\cX\times\cX}|f(x,y)|^2\dd\nu(x)\dd\nu(y)<\infty.
\end{align}
We define
\begin{align}\label{a1mu}
  \mu:=\E f(X_1,X_2).
\end{align}

We will mainly consider real-valued functions, but to apply functional
analysis it will sometimes
be convenient to also consider complex-valued fuctions.
When it is necessary to distinguish them, we use $\LLC(\cX)$ for the
complex Hilbert space of complex-valued functions $g$ on $\cX$ such that
$\int_{\cX}|g|^2\dd\nu<\infty$, and  
$\LLR(\cX)$ for the subspace of real-valued functions (this is 
a real Hilbert space).

All functions are assumed to be measurable.
We will sometimes omit ``a.s.''\ or ``a.e.''\ when this is obvious.

For a function $g:\cX\times\cX\to\bbC$, 
let $T_g$ denote the integral operator on $L^2(\cX)$ defined by 
\begin{align}\label{Tg}
  T_g\gf(x) := \int_{\cX}g(x,y)\gf(y)\dd\nu(y),
\qquad \gf\in L^2(\cX).
\end{align}
We will only consider the case when
$g\in L^2(\cX\times\cX)$; it is well-known that then,
for every $h\in L^2(\cX)$,
the integral in \eqref{Tg} converges for a.e.\ $x$ 
and that $T_gh\in L^2(\cX)$,
so that $T_g$ is well-defined, 
and furthermore that $T_g$ is a Hilbert--Schmidt operator on $L^2(\cX)$,
and thus in particular compact (and bounded).
(See \eg{} \cite[Proposition II.4.7 and Exercise  IX.2.19]{Conway}.)
We will also use the notation \eqref{Tg} for other measure spaces, in
particular $\cX^2=\cX\times\cX$ and
$\hcX=\cX\times\oi$.

Recall that a bounded operator $T$ on a Hilbert space $H$ is
\emph{self-adjoint}  (also called \emph{Hermitian}, or \emph{symmetric})
if $\innprod{Th,k}=\innprod{h,Tk}$ for all $h,k\in H$. It is easily seen
that if $g$ is real-valued and symmetric (i.e., $g(x,y)=g(y,x)$), then $T_g$
is self-adjoint.

Eigenvalues of operators are always counted with multiplicities; sets of
eigenvalues are thus in general really multisets.

We will occasionally
also use tensor notation for functions and operators.
If $g$ and $h$ are (real- or complex-valued) functions defined on measure 
spaces $\cY$ and $\cZ$, then
$g\tensor h$ denotes the function $(y,z)\mapsto g(y)h(z)$ defined on
$\cY\times\cZ$. 
It is well known that if $(\gf_\ga)_{\ga\in\cA}$ and
$(\psi_\gb)_{\gb\in\cB}$
are orthonormal bases in $L^2(\cY)$ and $L^2(\cZ)$,
then
$(\gf_\ga\tensor \psi_\gb)_{\ga\in\cA,\gb\in\cB}$ is an orthonormal basis in
$L^2(\cY\times\cZ)$.
If $g\in L^2(\cY\times\cY)$, and $h\in L^2(\cZ\times\cZ)$, then
$g\tensor h$ can be regarded as a kernel on $\cY\times\cZ$, and it is easily
seen that
\begin{align}\label{tensor}
T_{g\tensor h}(\gf\tensor \psi) = (T_g\gf)\tensor(T_h\psi),
\qquad 
\gf\in L^2(\cY), \psi\in L^2(\cZ).   
\end{align}
We write also $T_g \tensor T_h :=T_{g\tensor h}$.
(This is a special case of tensor products of linear operators on Hilbert,
or more general, spaces, but we have no need for the general theory.)
It follows from \eqref{tensor} that
if $\gf$ is an eigenfunction of $T_g$ with eigenvalue $\gl$ and
$\psi$ is an eigenfunction of $T_h$ with eigenvalue $\rho$,
then $\gf\tensor \psi$ is an eigenfunction of 
$T_g\tensor T_h=T_{g\tensor h}$ with eigenvalue $\gl\rho$.

For a function $f$ on $\cX\times\cX$, we define
its \emph{symmetric} and \emph{antisymmetric parts} by,
recalling \eqref{fi1}, 
\begin{align}\label{fsa}
  \fs:=\tfrac12(f+\xf)
\qquad\text{and}\qquad
\fa:=\tfrac12(f-\xf).
\end{align}
Thus $f=\fs+\fa$.
We also define (as another form of symmetrization)
the symmetric function $\hf$ on $\hcX\times\hcX$ by
(recall that $\hcX:=\cX\times\oi)$
\begin{align}\label{hff}
  \hf\bigpar{(x,t),(y,u)}&
:=
  \begin{cases}
    f(x,y), & t< u,
\\
    f(y,x), & t> u.
  \end{cases}
\end{align}
(For completeness, we may define $\hf\bigpar{(x,t),(y,u)}:=0$ when $t=u$;
this case has measure 0 and is therefore irrelevant.)
Note that, with $\hnu:=\nu\times\Leb$,
\begin{align}\label{ff}
&\int_{\hcX^2}|\hf(\hx,\hy)|^2\dd\hnu(\hx)\dd\hnu(\hy)
\notag\\&\qquad
=\int_{\oi^2}\int_{\cX^2}|\hf\bigpar{(x,t),(y,u)}|^2
 \dd\nu(x)\dd\nu(y)\dd t\dd u
\notag\\&\qquad
=\int_{\oi^2}\int_{\cX^2}|f(x,y)|^2\dd\nu(x)\dd\nu(y)\dd t\dd u
=\int_{\cX^2}|f(x,y)|^2\dd\nu(x)\dd\nu(y)
.\end{align}

In some cases (when studying $\Ump$ and $\Upm$) we need also the
corresponding antisymmetric function on $\hcX\times\hcX$ defined by 
\begin{align}\label{czf}
  \czf\bigpar{(x,t),(y,u)}&
:=
  \begin{cases}
    f(x,y), & t< u,
\\
    -f(y,x), & t> u.
  \end{cases}
\end{align}

For a real number $x$, we let $\floor{x}$ be $x$ rounded down to the nearest
integer. 
The complex unit is denoted $\ii$. (This should not be confused with $i$,
often used to denote indices.)
$\indic{\cE}$ denotes the indicator function of an event $\cE$.
The sign function $\sgn$ is given by
\begin{align}\label{sgn}
  \sgn(x):=
  \begin{cases}
    1,& x>0,\\
0,& x=0,\\
-1,& x<0.
  \end{cases}
\end{align}

We let $\dto$, $\pto$, and $\asto$ denote convergence of random variables 
in distribution, in probability, and almost surely, respectively.

Given a sequence $(a_n)_n$, we let 
$Y_n=\OLL(a_n)$ mean that $Y_n$ are random
variables such that
$\norm{Y_n}_{L^2}:=(\E [|Y_n|^2])\qq=O(a_n)$.
Similarly,
$Y_n=\oLL(a_n)$ means that 
$\norm{Y_n}_{L^2}=o(a_n)$, in other words that $Y_n/a_n\to 0 $ in $L^2$.

Unspecified limits are as \ntoo.

\subsection{Hoeffding's decomposition}
As said above, we assume $f\in L^2(\cX^2)$.
The basis of our work (as for many previous results for \Ustat{s})
is the orthogonal decomposition 
introduced (in the symmetric case) by
\citet{Hoeffding}. In the case $m=2$ treated here, 
the  orthogonal decomposition is:
\begin{align}\label{a2}
  f(x,y) = f_\emptyset + f_1(x) + f_2(y) + f_{12}(x,y)
\end{align}
where
\begin{align}\label{a3}
  f_\emptyset &:= \E f(X_1,X_2) = \int_{\cX\times\cX} f(x,y)\dd\nu(x)\dd\nu(y)
=\mu,
\\\label{a4}
  f_1(x) &:= \E f(x,\XX)-f_\emptyset = \int_{\cX} f(x,y)\dd\nu(y)-f_\emptyset,
\\\label{a5}
  f_2(y) &:= \E f(\XX,y)-f_\emptyset = \int_{\cX} f(x,y)\dd\nu(x)-f_\emptyset,
\\\label{a6}
f_{12}(x,y)&:=f(x,y)-f_1(x)-f_2(y)-f_\emptyset
.\end{align}
Equivalently, \eqref{a4}--\eqref{a6} can be written
\begin{align}
  f_1(X_1)&=\E [f(X_1,X_2)\mid X_1] - \E f(X_1,X_2),\label{a4E}
\\
  f_2(X_2)&=\E [f(X_1,X_2)\mid X_2] - \E f(X_1,X_2),\label{a5E}
\\
  f_{12}(X_1,X_2)&=f(X_1,X_2)-\E [f(X_1,X_2)\mid X_1] - \E [f(X_1,X_2)\mid X_2]
+\E f(X_1,X_2).\label{a6E}
\end{align}
Of course, \eqref{a6} makes \eqref{a2} trivial, but the point is that the 
four terms in the sum in \eqref{a2} are orthogonal in $L^2(\cX^2)$, which is
easily verified from the definitions \eqref{a3}--\eqref{a6E}, which imply
\begin{align}
  \E f_1(\XX)&=\E f_2(\XX)=0,\label{ai1}
\\
\E f_{12}(x,\XX)&=\E f_{12}(\XX,y)=\E f_{12}(X_1,X_2)=0. \label{ai2}
\end{align}

\begin{lemma}\label{LH}
  Let $(i,j)$ and $(k,l)$ be two pairs of indices with $i\neq j$, $k\neq l$,
  and $\set{i,j}\neq \set{k,l}$ 
(i.e., $(i,j)\neq(k,l)$ and $(i,j)\neq(l,k)$).
Then $f_{12}(X_i,X_j)$ and $f_{12}(X_k,X_l)$ are uncorrelated and thus
$\E\bigsqpar{f_{12}(X_i,X_j)f_{12}(X_k,X_l)}=0$.
\end{lemma}

\begin{proof}
A simple consequence of \eqref{ai2} and the 
standing assumption that $(X_i)$ are \iid
\end{proof}

\begin{remark}\label{Rsymm2}
  If $f$ is symmetric, then $f_1=f_2$, and furthermore, $f_{12}$ is symmetric.
On the other hand, if $f$ is antisymmetric, then $f_\emptyset=0$, 
$f_1=-f_2$, and $f_{12}$ is antisymmetric.  
\end{remark}

\subsection{Three distributions}\label{SS3}
The limit distributions below will, apart from normal distributions,
be given by (possibly infinite) linear combinations of 
independent copies of the following three random variables.

\begin{romenumerate}
\item \label{DD1}
If $\zeta\in N(0,1)$, then $\zeta^2$ has a $\chi^2(1)=\gG(\frac12,2)$
distribution. We will use the centred variable $\zeta^2-1$, which has mean 0,
variance
\begin{align}\label{zz1}
  \Var\bigpar{\zeta^2-1}=2,
\end{align}
and \chf
\begin{align}\label{zz2}
  \E e^{\ii t(\zeta^2-1)} = e^{-\ii t}(1-2\ii t)\qqw,
\qquad t\in\bbR.
\end{align}

\item \label{DD2}
\Levy's \emph{stochastic area}, which we denote by $\eta$, is 
the stochastic integral
\begin{align}\label{SI}
\eta:=\intoi B_1(x)\dd B_2(x)-\intoi B_2(x)\dd B_1(x)  
\end{align}
where $B_1(x)$ and $B_2(x)$ are two independent Brownian motions.
See e.g.\ \cite{Levy}, \cite{Yor}, 
\cite[Theorem II.43 and its Corollary]{Protter}.
For us this background is not important; we only need that
the stochastic area is a random variable $\eta$ 
with the \chf
\begin{align}\label{area}
  \E e^{\ii t\eta} = \frac{1}{\cosh(t)},
\qquad t\in\bbR.
\end{align}
The stochastic area has variance, from \eqref{SI} or from \eqref{area},
\begin{align}\label{etavar}
  \Var\eta=1,
\end{align}
and density $\frac{1}{2\cosh(\pi x/2)}$, $-\infty<x<\infty$.

\item\label{DD3} 
Thirdly, we define $\xeta$ by $\xeta:=\intoi B_1(x)\dd B_2(x)$,
i.e., ``the first half of the stochastic area \eqref{SI}''.
This random variable has the \chf,
see \eg{} \cite[(1)]{Yor}
(or as a consequence of the calculations in \refE{ESJ22} and \eqref{larea3}
below). 
\begin{align}\label{xeta}
  \E e^{\ii\xeta}=\frac{1}{\cosh\qq(t)},
\end{align}
and hence variance
\begin{align}\label{xetavar}
  \Var\xeta=\tfrac12.
\end{align}
\end{romenumerate}

\begin{remark}
The \chf{} \eqref{area} of the difference of the two stochastic integrals
in \eqref{SI} thus equals the product of their \chf{s}, which both are
\eqref{xeta}; however, the two integrals are \emph{not independent}.
Their joint \chf{} is given by, see  \cite[(1)]{Yor},
\begin{multline}\label{yor}
\E \Bigsqpar{\exp
 \Bigpar{\ii s\intoi B_1(x)\dd B_2(x)+\ii t\intoi B_2(x)\dd B_1(x)}}    
\\
=\Bigpar{\cosh^2\Bigparfrac{s-t}2
+\Bigparfrac{s+t}{s-t}^2\sinh^2\Bigparfrac{s-t}2^2}\qqw.
\end{multline}
\end{remark}

\begin{remark}\label{Rsum}
We will frequently use   sums
$\sumrR\gl_r(\zeta_r^2-1)$ where 
$(\gl_r)_1^R$ is a finite or infinite sequence of real numbers with 
$\sumrR\gl_r^2<\infty$, and
$\zeta_r\in N(0,1)$ are independent.
Note that this sum converges in $L^2$ (and a.s.), and is thus well defined
also when $R=\infty$.
Furthermore,
it is easily seen from \eqref{zz2} that 
the sum $\sumrR\gl_r(\zeta_r^2-1)$ has a characteristic function whose square
extends to a meromorphic function in the complex plane with poles
(counted with multiplicity)
at the points $1/(2\ii\gl_r)$ and nowhere else.
Thus the distribution of the sum determines the coefficients $(\gl_r)_1^R$ 
(up to order). 

By the representations in \refL{Larea} below, the same holds for sums
$\sumrR\gl_r\eta_r$  and $\sumrR\gl_r\xeta_r$.

It follows also that the distribution of the sum is \emph{not} normal,
since its \chf{} is not entire. 
This alternatively follows by the \Levy--\Cramer{} theorem \cite{Cramer}.
\end{remark}

Obviously, if $\xeta_1$ and $\xeta_2$ are \iid{} with the distribution
\eqref{xeta}, then $\xeta_1+\xeta_2$ and $\xeta_1-\xeta_2$ have the
distribution \eqref{area}.
The following lemma shows further relations between 
variables in \ref{DD1}--\ref{DD3} above.
 
\begin{lemma}\label{Larea}
let  $\zeta_{k,j}\in N(0,1)$ be independent. Then
\begin{align}
  \label{larea0}
\wxeta:=\sum_{k=-\infty}^\infty\frac{1}{(2k-1)\pi} \bigpar{\zeta^2_{k,1}-1}
\eqd\sumk \frac{1}{(2k-1)\pi} \bigpar{\zeta^2_{k,1}-\zeta^2_{k,2}}
\end{align}
has the distribution of $\xeta$ in \eqref{xeta}, and
\begin{align}  \label{larea}
\weta:=\sumk \frac{1}{(2k-1)\pi} 
\bigpar{\zeta^2_{k,1}+\zeta^2_{k,2}-\zeta^2_{k,3}-\zeta^2_{k,4}}
\end{align}
has the  distribution of the stochastic area $\eta$ in \eqref{area}.
\end{lemma}
\begin{proof}
All sums converge in $L^2$, \cf{} \refR{Rsum}.
We obtain the second equality (in distribution) in \eqref{larea0} by 
combining in the first sum  the terms for $k$ and $1-k$. 

The difference $\zeta^2_{k,1}-\zeta^2_{k,2}$
has the \chf, see \eqref{zz2},
\begin{align}\label{larea2}
  \E e^{\ii t (\zeta^2_{k,1}-\zeta^2_{k,2})}
=(1-2\ii t)\qqw(1+2\ii t)\qqw
=(1+4 t^2)\qqw,
\end{align}
and consequently, $\wxeta$ has the \chf
\begin{align}\label{larea3}
  \E e^{\ii t \wxeta}
=\prodk \Bigpar{1+\frac{4t^2}{(2k-1)^2\pi^2}}\qqw
=\bigpar{\cosh(t)}\qqw,
\end{align}\label{larea4}
where the last equality is well-known, see \eg{} \cite[(4.36.2)]{NIST}.
Since $\weta\eqd \wxeta+\wxeta'$ where $\wxeta'$ is an independent copy of $\wxeta$,
it follows that $\weta$ has the \chf
\begin{align}
  \E e^{\ii t \weta}
=\bigpar{\E e^{\ii t \wxeta}}^2
=\frac{1}{\cosh(t)},
\end{align}
which agrees with \eqref{area}.
\end{proof}

Formulas for the cumulants of these variables are given in \refApp{Acumulants}.

\section{Background: classical  \Ustat{s}}\label{SU}
As a background, we summarize in the following  theorem some
known result on the asymptotic distribution of $U$-statistics of the
standard type \eqref{U2}
in the special case of order $m=2$. 
The general case with arbitrary (fixed)
$m\ge2$ is similar with mainly notational complications;
the only essential difference is that degeneracies of higher order may occur
and then the limiting distributions are much more complicated
(although such cases are rarely seen in applications).
We restrict ourselves to $m=2$ because this is the case relevant for the
cyclic and alternating \Ustat{s} discussed in the present paper,
and also because it may be easier to see the general ideas in this somewhat
simpler case.
(We do not know any reference where all these results are collected and
presented for the case $m=2$.)
For completeness, we give a proof in \refApp{AA}.

For further results and for
the general case with arbitrary $m$, we refer to, for example,
\cite{Hoeffding,Hoeffding-LLN,RubinVitale1980,DynkinM} for the symmetric case
\eqref{U1},
and \cite[Chapter 11.1--2]{SJIII} for the general (asymmetric) case \eqref{U2}.
For the strong law of large numbers, see \refSS{SSstrong}.

In the theorem, note in particular the dichotomy between
the \emph{nondegenerate} case with $\gss>0$ and then variance of order
$n^3$ and asymptotically normal distribution (see \ref{TU1}),
and the \emph{degenerate} case in \ref{TU2}--\ref{TU2a} with $\gss=0$ 
and then variance of smaller order $n^2$ and a non-normal limit distribution.

\begin{theorem}\label{TU}
With notations and assumptions 
as in \refS{Sprel},
the following holds.
\begin{romenumerate}
\item \label{TU0}
We have
\begin{align}\label{tu00}
  \E U_n(f) = \tbinom n2\mu
\end{align}
and, as \ntoo, we have the weak law of large numbers
\begin{align}\label{tu0}
  \frac{1}{\binom n2} U_n(f)\pto \mu.
\end{align}

\item \label{TU1}
As \ntoo,
\begin{align}\label{gss1}
n^{-3}\Var[U_n(f)]\to
  \gss:=\tfrac{1}{3}
\bigpar{\E [f_1(\XX)^2] + \E [f_2(\XX)^2] + \E [f_1(\XX)f_2(\XX)]}
\end{align}
and 
\begin{align}\label{tu1}
  n^{-3/2}\lrpar{U_n(f)-\tbinom n2 \mu} \dto N(0,\gss).
\end{align}
Furthermore, $\gss>0$ unless $f_1(\XX)=f_2(\XX)=0$ a.s.

\item\label{TU2} 
If\/ $f_1(\XX)=f_2(\XX)=0$ a.s., and thus $\gss=0$, then
\begin{align}
  \label{tu20}
  \Var[U_n(f)] = \tbinom{n}{2}\Var [f(X_1,X_2)]
=\tfrac{1}{2}n^2\Var [f(X_1,X_2)] + O(n).
\end{align}
Moreover, there exists
a finite or infinite sequence of real numbers $(\gl_r)_1^R$ such that
\begin{align}\label{tu2}
  n\qw\lrpar{U_n(f)-\tbinom n2 \mu} \dto W:=\sumrR\tfrac12\gl_r(\zeta_r^2-1),
\end{align}
where $(\zeta_r)_1^R$ are independent standard normal variables.
The coefficients
$(\gl_r)_1^R$ 
are the nonzero eigenvalues (with multiplicities) 
of the self-adjoint 
integral operator $T_\hfmu$ 
on $L^2(\cX\times\oi,\nu\times\Leb)$
(where $\Leb$ is Lebesgue measure)
defined as in \eqref{Tg} using \eqref{hff}.
We have
\begin{align}\label{tu22}
\Var W=
 \tfrac12 \sumrR\gl_r^2 = \tfrac12\Var[f(X_1,X_2)]
<\infty.
\end{align}

\item \label{TU2s}
In the special case of \ref{TU2} where
furthermore
$f$ is symmetric, 
the coefficients
$(\gl_r)_1^R$ in \eqref{tu2} are the nonzero eigenvalues (with multiplicities) 
of the self-adjoint 
integral operator $T_\fmu$ 
on $L^2(\cX,\nu)$.

\item \label{TU2a}
In the special case of \ref{TU2} where
furthermore
$f$ is antisymmetric, 
then also
\begin{align}
  \label{tu2a}
  n\qw{U_n(f)} \dto 
W:=\sumqQp\gla_q\eta_q,
\end{align}
where $(\eta_q)_1^\Qp$ are independent random variables with the stochastic area
distribution \eqref{area}, and
the coefficients 
$(\gla_q)_1^\Qp$  are the positive numbers such that the imaginary
number $\ii\gla_q$ is an eigenvalue of the 
anti-self-adjoint
operator $T_f$ on $\LLC(\cX,\nu)$.
We have
\begin{align}\label{tu2a2}
  \Var W = \sumqQp(\gla_q)^2 = \tfrac12\Var[f(X_1,X_2)].
\end{align}

\end{romenumerate}
\end{theorem}

As said above, the proof is given in \refApp{AA}.

\begin{remark}\label{Rsanity}
As a sanity check, we note that if $f$ is symmetric, then the 
nonzero eigenvalues of the operators $T_{\hf-\mu}$ in \ref{TU2} and
$T_{f-\mu}$ in \ref{TU2s} are the same, so the conclusions agree.
In fact, if $f$ is symmetric, the \eqref{hff} yields 
\begin{align}
  \label{hffsymm}
\hf\bigpar{(x,t),(y,u)}=f(x,y).
\end{align}
Letting $\etta$ denote the function on $\oi^2$ that is constant 1,
we thus have, using the tensor notation in \eqref{tensor},
$\hf=f\tensor\etta$ and consequently
\begin{align}
  T_{\hf}=T_f\tensor T_{\etta}.
\end{align}
$T_{\etta}$ is the integral operator $T_{\etta} g(t)=\intoi g(u)\dd u$; this is
the projection onto the constant functions and has a single nonzero
eigenvalue 1. Consequently, $T_{\hf}=T_f\tensor T_{\etta}$ has the same
nonzero eigenvalues as $T_f$. 
(In this simple case, this can also easily be seen
directly from \eqref{hffsymm}; the eigenfunctions of $T_\hf$ with nonzero
eigenvalues are the functions of the form $\gf(x,t)=\gf_1(x)$ where $\gf_1$
is an eigenfunction of $T_f$ with the same eigenvalue.) 
\end{remark}

\section{Cyclic \Ustat{s}}\label{SUc}

We next give the corresponding result for cyclic \Ustat{s}.

\begin{theorem}\label{TC}
With notations and assumptions 
as in \refS{Sprel},
the following holds for the cyclic \Ustat{} $\Uc_n(f)$ in \eqref{Uc}.
\begin{romenumerate}
\item \label{TC0}
We have
\begin{align}\label{tc00}
  \E [\Uc_n(f)] = n \floor{\tfrac{n}2}\mu
=\tfrac{n^2}{2}\mu +O(n),
\end{align}
and, as \ntoo, we have the weak law of large numbers 
\begin{align}\label{tc0}
  \frac{2}{n^2} \Uc_n(f)\pto \mu.
\end{align}

\item \label{TC1}
As \ntoo,
\begin{align}\label{gss1c}
n^{-3}\Var[\Uc_n(f)]\to
  \gss&:=
\tfrac{1}{4}\Var\bigsqpar{f_1(\XX)+f_2(\XX)}
\notag\\&\phantom:
=\tfrac14\bigpar{\E [f_1(\XX)^2] + \E [f_2(\XX)^2] + 2\E [f_1(\XX)f_2(\XX)]}
\end{align}
and
\begin{align}\label{tc1}
  n^{-3/2}\lrpar{\Uc_n(f)-\tfrac{n^2}{2}\mu} \dto N(0,\gss),
\end{align}
Furthermore, $\gss>0$ unless $f_1(\XX)+f_2(\XX)=0$ a.s.

  \item\label{TC2} 
If\/ $f_1(\XX)+f_2(\XX)=0$ a.s., 
and thus $\gss=0$, 
then
\begin{align}\label{tc20}
  \Var[\Uc_n(f)] = n \floor{\tfrac{n}2}\Var [f_{12}(X_1,X_2)]
=\tfrac{1}{2}n^2\Var [f_{12}(X_1,X_2)] + O(n).
\end{align}
Moreover, there exist
finite or infinite sequences of real numbers $(\gls_r)_1^{R}$ 
and $(\gla_q)_1^{\Qp}$ 
such that
\begin{align}\label{tc2}
n\qw\lrpar{\Uc_n(f)-\E[\Uc_n(f)]} \dto 
W:=\sumrR\tfrac12\gls_r(\zeta_r^2-1)
+\sumqQp\gla_q \eta_q
\end{align}
where $(\zeta_r)_1^R$ are  standard normal variables and $(\eta_q)_1^\Qp$ 
have the
stochastic area distribution \eqref{area}, and all are independent.
The coefficients
$(\gls_r)_1^R$ in \eqref{tc2} are the nonzero eigenvalues (with
multiplicities) 
of the self-adjoint 
integral operator $T_{\fiiis}$ 
on $L^2(\cX,\nu)$,
where, recalling \eqref{fsa},
$\fiiis:=(f_{12}+\xf_{12})/2$ is the symmetric part of $f_{12}$.
Similarly,
the coefficients $(\gla_q)_1^\Qp$ 
are the positive numbers such that the imaginary
number $\ii\gla_q$ is an eigenvalue of the 
anti-self-adjoint operator $T_\fiiia$ on $\LLC(\cX,\nu)$,
where $\fiiia:=(f_{12}-\xf_{12})/2$ is the antisymmetric part of $f_{12}$.
We have
\begin{align}\label{Cemms}
  \sumrR(\gls_r)^2&=\Var[\fiiis(X_1,X_2)],
\\\label{Cemma}
  \sumqQp(\gla_q)^2&=\tfrac12\Var[\fiiia(X_1,X_2)],
\end{align}
and                    
\begin{align}\label{Cemmas}
\Var W&
=\tfrac12 \sumrR(\gls_r)^2+  \sumqQp(\gla_q)^2
=\tfrac12\Var[f_{12}(X_1,X_2)].
\end{align}

\end{romenumerate}
\end{theorem}

\begin{proof}
We start by closely following the proof in \refApp{AA}
of the corresponding classical result for the usual \Ustat{s} in \refT{TU}.

Note that the cyclic \Ustat{} $\Uc_n$ in \eqref{Uc} is a sum of 
$n\flfracnn= \frac12n^2+O(n)$ terms.
This shows \eqref{tc00}.

We substitute the decomposition \eqref{a2} into the definition \eqref{Uc}.
Note that each $i\in\bbZ_n$ occurs in $\flfracnn$ terms in the 
double sum \eqref{Uc}, and so does every value of $i+j\in\bbZ_n$.
Hence,
\begin{align}\label{aida}
\Uc_n(f)
= \E [\Uc_n(f)] + \lrflfracnn \sumin f_1(X_i)  + \lrflfracnn \sumin f_2(X_i)
+ 
\sum_{i\in\bbZ_n}\sum_{1\le j < n/2} f_{12}(X_i,X_{i+j}).
\end{align}
 Each of the first two sums is a sum of \iid{} random variables
with mean zero and finite variance; 
hence these sums are $\OLL(n\qq)$, and the corresponding terms
are $\OLL(n^{3/2})$. Furthermore, the final double sum is a sum of $O(n^2)$
identically distributed 
terms that are uncorrelated by \refL{LH} and have mean zero, and thus
the double sum is $\OLL(n)=\oLL(n^{3/2})$.

In the rest of the proof we replace $f$ by $f-\mu$
(which does not change $f_1$, $f_2$, or $f_{12}$); hence we may and do
assume without loss of generality that $\mu=0$. Thus $\E[\Uc_n(f)]=0$ by
\eqref{tc00}.

\pfitemx{\ref{TC0} and \ref{TC1}}
By \eqref{aida} and the comments after it,
\begin{align}\label{aid1}
  n^{-3/2}\Uc_n(f)
= n^{-3/2}\frac n2 \sumin\bigpar{f_1(X_i)+ f_2(X_i)} + \oLL(1).
\end{align}
Since the variables $X_i$ are \iid, \eqref{aid1}
implies immediately both \eqref{gss1c} and,
by the classical central limit theorem 
(together with the Cram\'er--Slutsky theorem \cite[Theorem 5.11.4]{Gut}),
\eqref{tc1}.
Any of these implies \eqref{tc0}, and \eqref{tc00} was proved above.

\pfitemref{TC2}
Since $f_1(X_i)+f_2(X_i)=0$ a.s., the first two sums in \eqref{aida} cancel.
Hence, using also our simplifying assumption $\mu=0$, we now have 
$\Uc_n(f)=\Uc_n(f_{12})$, and we may simplify the notation by assuming
$f=f_{12}$.
Then, \refL{LH} shows that \eqref{Uc} is a sum of $n\flfracnn$
uncorrelated,
identically distributed,
terms, and \eqref{tc20} follows.

It will be convenient to consider even $n$, so we first note that for
any $n\ge1$,
by \eqref{Uc} and some bookkeeping,
\begin{align}\label{aid3}
  \Uc_{2n+1}(f) = \Uc_{2n}(f) 
+ \sumin f(X_i,X_{i+n})
+ \sumin f(X_{2n+1},X_{i})
+\sum_{i=n+1}^{2n}f(X_i,X_{2n+1}).
\end{align}
(As a check, note that the total number of terms is $(2n+1)n$ in $\Uc_{2n+1}$
and $2n(n-1)$ in $\Uc_{2n}$, and that every term in $\Uc_{2n}$ appears also
in $\Uc_{2n+1}$.)
Each of the three sums in \eqref{aid3} is,
by \refL{LH}, a sum of $n$ uncorrelated variables, and it follows that it is
$\OLL(n\qq)$. Hence it suffices to prove \eqref{tc2} for even $n$; the case
of odd $n$ then follows. 

We consider thus $\Uc_{2n}$, where we assume $n\ge2$.
Define 
\begin{align}\label{CtX}
\tX_i:=(X_{i},X_{i+n}),
\qquad i=1,\dots,n.  
\end{align}
Then $(\tX_i)\xoon$ is an \iid{} sequence of random variables in $\cX^2$.
Define the function $F$ on $\cX^4=\cX^2\times\cX^2$ by
\begin{align}\label{Fc}
  F\bigpar{(x_1,x_2),(y_1,y_2)}
:=f(x_1,y_1)+f(y_1,x_2)+f(x_2,y_2)+f(y_2,x_1).
\end{align}
It now follows from the definitions \eqref{Uc} and \eqref{U2} that
\begin{align}\label{CUU}
  \Uc_{2n}(f) = U_n(F;\tX_1,\dots,\tX_n).
\end{align}
Consequently, we may use the classical result
\refT{TU} for the
usual \Ustat{s}, applied to $F$ and $(\tX_i)$.
%
Recall that we have $f_1=f_2=0$ and that we have assumed $\mu=0$,
which clearly implies also
(using that $(X_i)_i$ are \iid)
\begin{align}
  F_\emptyset 
=\E F(\tX_1,\tX_2)
= \E F\bigpar{(X_1,X_2),(X_3,X_4)}=0.
\end{align}
Furthermore, \eqref{a4} applied to $F$ yields, using \eqref{Fc},
\begin{align}\label{FF}
  F_1(x_1,x_2) &
= \E F\bigpar{(x_1,x_2),(X_3,X_4)}
\notag\\&
=\E f(x_1,X_3)+\E f(X_3,x_2) + \E f(x_2,X_4) + \E f(X_4,x_1)
\notag\\&
= f_1(x_1)+f_2(x_2)+f_1(x_2)+f_2(x_1)
=0
\end{align}
and similarly $F_2(x_1,x_2)=0$.
(This follows also from \eqref{CUU} and \eqref{tc20}, which show that 
$\Var[ U_n(F)]=\Var[\Uc_{2n}(f)]=O(n^2)$, together with \eqref{gss1}.)

Hence, \refT{TU}\ref{TU2} applies, and shows that \eqref{tu2} holds for
$U_n(F)$; it remains to find the eigenvalues $\gl_r$ of $T_{\hF}$,
where $T_{\hF}$ is the integral operator on $L^2(\cX^2\times\oi)$ 
with kernel,
by \eqref{hff} and \eqref{Fc},
\begin{align}\label{Cjw1}
&\hF\bigpar{(x_1,x_2,t),(y_1,y_2,u)} 
= \bigpar{f(x_1,y_1)+f(y_1,x_2)+f(x_2,y_2)+f(y_2,x_1)}\indic{t<u}
\notag\\&\qquad
+
\bigpar{f(y_1,x_1)+f(x_1,y_2)+f(y_2,x_2)+f(x_2,y_1)}\indic{t>u}
\notag\\&
=\hf\bigpar{(x_1,t),(y_1,u)}
+
\hf\bigpar{(y_1,t),(x_2,u)}
+
\hf\bigpar{(x_2,t),(y_2,u)}
+
\hf\bigpar{(y_2,t),(x_1,u)}
\notag\\&
=\hf\bigpar{(x_1,t),(y_1,u)}
+
\hxf\bigpar{(x_2,t),(y_1,u)}
+
\hf\bigpar{(x_2,t),(y_2,u)}
+
\hxf\bigpar{(x_1,t),(y_2,u)}.
\end{align}
We pause for a general observation on this type of kernels.
Let 
\begin{align}\label{Cjw2}
  L^2_0(\cX\times\oi)
:=\Bigset{h\in L^2(\cX\times\oi):\int_\cX h(x,t)\dd\nu(x)=0 \text{ for
  a.e.\ }t\in\oi} 
\end{align}
and let, for $j=1,2$, $M_j$ be the subspace of $L^2(\cX\times\cX\times\oi)$
consisting of functions of the type $g(x_1,x_2,t)=h(x_j,t)$ for some 
$h\in L^2_0(\cX\times\oi)$. 

If $g_j\in M_j$ for $j=1,2$, then with obvious notation,
\begin{align}\label{Cjw3}
&
\int_{\cX\times\cX\times\oi}g_1(x_1,x_2,t)g_2(x_1,x_2,t)\dd\nu(x_1)\dd\nu(x_2)\dd t
\notag\\&\qquad
=
\int_{\cX\times\cX\times\oi}h_1(x_1,t)h_2(x_2,t)\dd\nu(x_1)\dd\nu(x_2)\dd t
=0.
\end{align}
Thus $M_1$ and $M_2$ are orthogonal subspaces of $L^2(\cX\times\cX\times\oi)$.
Let $M_1\oplus M_2$ by their direct sum; this is also a subspace of
$L^2(\cX\times\cX\times\oi)$.

\begin{lemma}\label{LXX}
Let $i,j\in\set{1,2}$.
  Suppose that $g$ is a function in $L^2\bigpar{(\cX\times\cX\times\oi)^2}$ 
of the form
$g\bigpar{(x_1,x_2,t),(y_1,y_2,u)}=h(x_i,y_j,t,u)$ where 
$\int_\cX h(x,y,t,u)\dd\nu(x)=0$ for a.e.\ $y\in\cX$ and $t,u\in\oi$.
Then $T_g$ maps $L^2(\cX\times\cX\times\oi)$ into the subspace $M_i$.
Furthermore, $T_g$ maps $M_{3-j}$ to $0$.
\end{lemma}
\begin{proof}
  Simple consequences of the definitions and Fubini's theorem.
\end{proof}

To continue the proof of \refT{TC},
we see from \eqref{Cjw1} that $\hF$ is a sum of 8 terms, each of them of the
type in \refL{LXX}.
Hence, $T_\hF$ maps $L^2(\cX\times\cX\times\oi)$ into $\MM$, and thus all
eigenfunctions for a nonzero eigenvalue belong to $\MM$. Hence, to find the
nonzero eigenvalues, it suffices to consider the restriction of 
$T_\hF$ to $\MM$.

Let $\psi_1+\psi_2\in \MM$, with $\psi_j\in M_j$, and let 
(with a minor abuse of notation) 
$\psi_j$ denote also 
the corresponding function in $L^2_0:=L^2_0(\cX\times\oi)$. 
Then, by \eqref{Cjw1} and \refL{LXX} (which also shows that some terms vanish),
\begin{align}\label{Cjw4}
  T_\hF(\psi_1+\psi_2)(x_1,x_2,t)
=T_\hf \psi_1(x_1,t) + T_\hxf\psi_1(x_2,t)
+T_\hf \psi_2(x_2,t) + T_\hxf\psi_2(x_1,t),
\end{align}
where the four terms on the \rhs{} belong to $M_1$, $M_2$, $M_2$, $M_1$,
respectively.
Since $\psi_1+\psi_2$ is an eigenfunction with eigenvalue
$\gl$ if and only if
the \lhs{} of \eqref{Cjw4} equals $\gl\psi_1(x_1,t) + \gl\psi_2(x_2,t)$,
it follows by separating both sides of \eqref{Cjw4} into their components in
$M_1$ and $M_2$ (or, equivalently, by separating terms depending on $x_1$
from terms depending on $x_2$)
that
$\psi_1+\psi_2$ is an eigenfunction with eigenvalue $\gl\neq0$ if and only if
\begin{align}\label{Cjw5}
  \begin{cases}
T_\hf \psi_1  + T_\hxf\psi_2 = \gl\psi_1,
\\
T_\hxf\psi_1 +T_\hf \psi_2 = \gl\psi_2. 
  \end{cases}
\end{align}
By adding and subtracting these equations, 
we obtain the equivalent system
\begin{align}\label{Cjw6}
  \begin{cases}
(T_\hf+T_\hxf)( \psi_1 +\psi_2) = \gl(\psi_1+\psi_2),
\\
(T_\hf-T_\hxf)(\psi_1-\psi_2) = \gl(\psi_1-\psi_2)  . 
  \end{cases}
\end{align}
Let, for $0\neq\gl\in\bbC$ and an operator $T$ on a vector space,
$E_\gl(T)$ denote the eigenspace $\set{h:Th=\gl h}$.
The map $\psi_1+\psi_2\mapsto(\psi_1+\psi_2,\psi_1-\psi_2)$
is a bijection of $M_1\oplus M_2$ onto $L^2_0\times L^2_0$, and
\eqref{Cjw6} shows that this bijection maps the eigenspace 
$E_\gl(T_\hF)$ onto $E_\gl(T_{\hf+\hxf})\oplus E_\gl(T_{\hf-\hxf})$.
In particular, the dimensions agree, which shows that the multiset of
nonzero eigenvalues of 
$T_\hF$ equals the union of the multisets of nonzero eigenvalues of  
$T_{\hf+\hxf}$ and $T_{\hf-\hxf}$.
We analyze these separately.

First, recalling \eqref{fi1} and \eqref{fsa},
$f+\xf=2\fs$ is symmetric, and thus, by \eqref{hff},
$\widehat{f+\xf}(x,y,t,u)=2\widehat{\fs}(x,y,t,u)=2\fs(x,y)$.
Hence, the corresponding eigenvalues are 2 times the eigenvalues $\gls_r$
of $T_\fs$.
(Cf.\ \refR{Rsanity}.)
The contribution from the eigenvalues of $T_{\hf+\hxf}$ to the limit (in
distribution) \eqref{tu2} of $n\qw U_n(F)$ is thus
\begin{align}\label{Cemm1}
  \sumrR\tfrac12(2 \gls_r)(\zeta_r^2-1)
=
  \sumrR \gls_r(\zeta_r^2-1)
.\end{align}

On the other hand,
$f-\xf=2\fa$ is antisymmetric. Its eigenvalues on the positive
imaginary axis are  $(2\ii\gla_q)_1^\Qp$, and thus it follows from
\refL{Lanti} that 
the contribution from the eigenvalues of $T_{\hf-\hxf}$ to the limit
\eqref{tu2} is 
\begin{align}\label{Cemm2}
  \sumqQp2\gla_q\eta_q.
\end{align}
It follows from \eqref{tu2} also that the contributions in \eqref{Cemm1} amd
\eqref{Cemm2} are independent.

Consequently,
recalling \eqref{CUU}, \eqref{tu2} for $U_n(F)$ implies that
\begin{align}\label{Cemm3}
  \frac{1}{n} \Uc_{2n}(f)=\frac{1}{n} U_n(F)
\dto \sumrR \gls_r(\zeta_r^2-1) + 2\sumqQp\gla_q\eta_q
.\end{align}
This shows that 
$\frac{1}{2n} \Uc_{2n}(f)$ has the limit $W$ in \eqref{tc2}.
In other words, \eqref{tc2} holds for even $n$, which as said above implies
the general case.

Finally, \eqref{Cemms} follows from \eqref{lts2} applied to 
$\fs$,
and \eqref{Cemma} follows from \eqref{lanti3} applied to $\fa$.
The first equality in \eqref{Cemmas} follows from \eqref{tc2};
the second follows from \eqref{Cemms} and \eqref{Cemma} since $\fs$ and
$\fa$ are orthogonal.
\end{proof}

\begin{remark}\label{RC=U}
  If $f$ is symmetric, then we obtain the same asymptotic results for
  $\Uc_n$ as in  \refT{TU} for $U_n$. This is nothing new, since, as noted
  in \refR{Rsymm}, in this case $\Uc_n=U_n$ for odd $n$, and
$\Uc_n=U_n+\OLL(n\qq)$ for even $n$.

On the other hand, if $f$ is antisymmetric, then $f_2(x)=-f_1(x)$, and thus
$\Uc_n$ is always of the degenerate type, while $U_n$ is nondegenerate
unless $f_1(x)=0$. In the latter case, when $f$ is antisymmetric and $f=f_{12}$,
we again find the same asymptotic results for $\Uc_n$ and $U_n$, this time
less obviously.
\end{remark}

\begin{remark}\label{Rdecouple}
Note that, rather surprisingly, \eqref{tc2} shows that in the degenerate
case \ref{TC2}, the contributions to $\Uc_n(f)$ from the 
symmetric and antisymmetric parts of $f$ decouple, so that $W$ is a sum of
two independent components.
Equivalently, by the Cram\'er--Wold device and
applying the theorem to $s\fs+t\fa$ for $s,t\in\bbR$,
$n\qw\Uc_n(\fs)$ and $n\qw\Uc_n(\fa)$ converge jointly in distribution to 
the two independent sums in \eqref{tc2}.

There is no such decoupling for the standard \Ustat{} $U_n$, or for any of
the alternating \Ustat{s}, as will be seen in \refE{ESJ22}.
Hence, the decoupling for $\Uc_n$ seems to be an effect of the larger
(cyclical) symmetry of $\Uc_n$.
\end{remark}

\section{Bi-alternating \Ustat{s}}\label{SB}

We next give the corresponding result for the bi-alternating \Ustat{}
$\Umm$.
The result is similar to \refTs{TU} and \ref{TC}, but the alternating signs
in the definition \eqref{U--} lead to cancellations and as a result there
is no case corresponding to the nondegenerate case for standard or cyclic
\Ustat{s}; the main case corresponds to the degenerate case in the previous
theorems.
There is also a rather uninteresting new case \ref{TBv}, included for
completeness, 
with an even smaller variance $O(n)$
and
$\Umm_n(f)$ reduced to a sum of \iid{} variables.  

\begin{theorem}\label{TB}
With notations and assumptions 
as in \refS{Sprel},
the following holds for the bi-alternating \Ustat{} $\Umm_n(f)$ in \eqref{U--}.
\begin{romenumerate}
\item \label{TB0}
We have
\begin{align}\label{tb00}
  \E [\Umm_n(f)] = -\floor{\tfrac{n}{2}} \mu
=O(n),
\end{align}
and, as \ntoo, we have the weak law of large numbers 
\begin{align}\label{tb0}
  \frac{1}{\binom n2} \Umm_n(f)\pto 0.
\end{align}

  \item\label{TB2} 
We have
\begin{align}\label{tb20}
  \Var[\Umm_n(f)] 
=\tfrac{1}{2}n^2\Var [f_{12}(X_1,X_2)] + O(n).
\end{align}
Moreover, there exists a
finite or infinite sequence of real numbers $(\gl_r)_1^{R}$ 
such that
\begin{align}\label{tb2}
n\qw\lrpar{\Umm_n(f)-\E[\Umm_n(f)]} \dto 
W:=\sumrR\tfrac12\gl_r(\zeta_r^2-1)
\end{align}
where $(\zeta_r)_1^R$ are  independent standard normal variables.
The coefficients
$(\gl_r)_1^R$ in \eqref{tb2} are the nonzero eigenvalues (with multiplicities) 
of the self-adjoint 
integral operator $T_{\hfiii}$ 
on $L^2(\cX\times\oi,\nu\times\Leb)$.
We have
\begin{align}\label{Bemms}
\Var W = \tfrac12  \sumrR\gl_r^2&
=\tfrac12\Var[f_{12}(X_1,X_2)].
\end{align}

\item \label{TB2s}
In the special case of \ref{TB2} where
furthermore
$f$ is symmetric, 
the coefficients
$(\gl_r)_1^R$ in \eqref{tb2} are the nonzero eigenvalues (with multiplicities) 
of the self-adjoint 
integral operator $T_{f_{12}}$ 
on $L^2(\cX,\nu)$.

\item \label{TB2a}
In the special case of \ref{TB2} where
furthermore
$f$ is antisymmetric, 
then also
\begin{align}
  \label{tb2a}
  n\qw\Umm_n(f) \dto 
W:=\sumqQp\gla_q\eta_q,
\end{align}
where $(\eta_q)_1^\Qp$ are independent random variables with the stochastic area
distribution \eqref{area}, and
the coefficients $(\gla_q)_1^\Qp$  are the positive numbers such that the
imaginary number $\ii\gla_q$ is an eigenvalue of the 
anti-self-adjoint operator $T_{f_{12}}$ on $\LLC(\cX,\nu)$.
We have
\begin{align}\label{tb2a2}
  \Var W = \sumqQp(\gla_q)^2 = \tfrac12\Var[f_{12}(X_1,X_2)].
\end{align}
\item \label{TBv}
If $f_{12}=0$, then
\begin{align}\label{tbv01}
  \Var[\Umm_{2n}(f)]&=2n\gsse+O(1),
\\  \label{tbv02}
  \Var[\Umm_{2n+1}(f)]&=(2n+1)\gsso+O(1),
\end{align}
where
\begin{align}\label{tbv3}
  \gsse&:=\tfrac12\bigpar{\Var[f_1(X)] + \Var[f_2(X)]},
\\\label{tbv4}
  \gsso&:=\tfrac12\Var[{f_1(X)+f_2(X)}].
\end{align}
Furthermore,
\begin{align}\label{tbv91}
  (2n)\qqw\bigpar{\Umm_{2n}(f)-\E [\Umm_{2n}(f)]}&\dto N(0,\gsse),
\\  \label{tbv92}
(2n+1)\qqw\bigpar{\Umm_{2n+1}(f)-\E[\Umm_{2n+1}(f)]}&\dto N(0,\gsso).
\end{align}
\end{romenumerate}
\end{theorem}

\begin{proof}
We follow the proofs of \refTs{TU} and \ref{TC}, with some differences.
First, \eqref{tb00} follows immediately from \eqref{U--} and
\begin{align}
  \sum_{1\le i<j\le n} (-1)^{i+j}
=
  \sumjn (-1)^{j}\sum_{i=1}^{j-1}(-1)^i
=  \sumjn (-1)^{j+1}{\indic{\text{$j$ is even}}}
=-\lrfloor{\frac n2}.
\end{align}

We substitute the decomposition \eqref{a2} into the definition \eqref{U--},
and obtain by simple calculations
\begin{align}\label{Baida}
\Umm_n(f)&
= \E[\Umm_n(f)] - \sumin \indic{\text{$n-i$ is odd}}f_1(X_i)  
- \sumjn \indic{\text{$j$ is even}}f_2(X_j) 
\notag\\&\qquad{}
+ 
\sum_{1\le i< j \le n}(-1)^{i+j} f_{12}(X_i,X_{j}).
\end{align}
 Each of the first two sums is a sum of $\flfrac n2$ \iid{} random variables
with mean zero and finite variance; 
hence these sums are $\OLL(n\qq)$.
The final double sum is a sum of $\binom n2$
terms that are uncorrelated by \refL{LH} and have mean zero, and up to sign
have the same distribution; hence the double sum has
variance $\binom n2 \Var[f_{12}(X_1,X_2)]$.
It is also easily seen that the double sum is orthogonal to the two other sums,
and \eqref{tb20} follows.

The weak law of large numbers \eqref{tb0} is a consequence of \eqref{tb00}
and \eqref{tb20}. 

\pfitemref{TB2}
As in the other proofs, we  replace $f$ by $f-\mu$. 
In the rest of the proof we thus may assume that $\mu=0$, and thus 
$\E \Umm_n=0$.
By \eqref{Baida}, then
\begin{align}
  \Umm_n(f) = \Umm_n(f_{12})+\OLL(n\qq),
\end{align}
so it suffices to consider $\Umm(f_{12})$, and we may, without loss of
generality, for (notational) simplicity assume $f=f_{12}$.

Again it will be convenient to consider even $n$, 
and we note that
\eqref{U--} implies
\begin{align}\label{Baid3}
  \Umm_{2n+1}(f) = \Umm_{2n}(f) 
+ \sum_{i=1}^{2n} (-1)^{i+1} f(X_i,X_{2n+1})
.\end{align}
The sum in \eqref{Baid3} is,
by \refL{LH} and our assumption $f=f_{12}$, 
a sum of $2n$ uncorrelated variables with means 0, and it follows that it is
$\OLL(n\qq)$. Hence it suffices to prove \eqref{tb2} for even $n$.

We consider thus $\Umm_{2n}$, where we assume $n\ge2$.
We now define (cf.\ \eqref{CtX} for the cyclic \Ustat)
\begin{align}\label{BtX}
\tX_i:=(X_{2i-1},X_{2i}),
\qquad i=1,\dots,n.  
\end{align}
Again, $(\tX_i)\xoon$ is an \iid{} sequence of random variables in $\cX^2$.
We now define the function $F$ on $\cX^4=\cX^2\times\cX^2$ by
\begin{align}\label{BF}
  F\bigpar{(x_1,x_2),(y_1,y_2)}
:=f(x_1,y_1)-f(x_1,y_2)-f(x_2,y_1)+f(x_2,y_2),
\end{align}
and
it follows from the definitions \eqref{U--} and \eqref{U2} that
\begin{align}\label{BUU}
  \Umm_{2n}(f) = U_n(F;\tX_1,\dots,\tX_n)
-\sumin f(X_{2i-1},X_{2i})
.\end{align}
The final sum in \eqref{BUU} is $\OLL(n\qq)$, as a sum of $n$ \iid{}
variables with zero mean, and is thus negligible in \eqref{tb2}.
Consequently, we may use \refT{TU} for the
usual \Ustat{s}, applied to $F$ and $(\tX_i)$.
Recall that we have assumed $f=f_{12}$ and thus $f_1=f_2=0=\mu$,
which clearly implies also
\begin{align}
  F_\emptyset 
=\E F(\tX_1,\tX_2)
= \E F\bigpar{(X_1,X_2),(X_3,X_4)}=0.
\end{align}
Furthermore, \eqref{a4}--\eqref{a5} applied to $F$ yield, similarly to
\eqref{FF}, 
\begin{align}\label{BFF}
  F_1(x_1,x_2) 
= F_2(x_1,x_2)
=0
.\end{align}
(Again, this follows also from \eqref{BUU} and \eqref{tb20}, which show that 
$\Var[ U_n(F)]=O(n^2)$, together with \eqref{gss1}.)

Hence, \refT{TU}\ref{TU2} applies, and shows that \eqref{tu2} holds for
$U_n(F)$. It remains to find the eigenvalues $\gl_r$ of $T_{\hF}$,
where $T_{\hF}$ now is the integral operator on $L^2(\cX^2\times\oi)$ 
with kernel,
by \eqref{hff} and \eqref{BF},
\begin{align}\label{Bjw1}
&\hF\bigpar{(x_1,x_2,t),(y_1,y_2,u)} 
= \bigpar{f(x_1,y_1)-f(x_1,y_2)-f(x_2,y_1)+f(x_2,y_2)}\indic{t<u}
\notag\\&\qquad
+
\bigpar{f(y_1,x_1)-f(y_1,x_2)-f(y_2,x_1)+f(y_2,x_2)}\indic{t>u}
\notag\\&
=\hf\bigpar{(x_1,t),(y_1,u)}
-
\hf\bigpar{(x_1,t),(y_2,u)}
-
\hf\bigpar{(x_2,t),(y_1,u)}
+
\hf\bigpar{(x_2,t),(y_2,u)}
.\end{align}
It follows again from \refL{LXX}
that $T_\hF$ maps $L^2(\cX\times\cX\times\oi)$ into $\MM$, and thus all
eigenfunctions for a nonzero eigenvalue belong to $\MM$. 

Let $\psi_1+\psi_2\in \MM$, with $\psi_j\in M_j$, and let again
(with a minor abuse of notation) 
$\psi_j$ denote also 
the corresponding function in $L^2_0:=L^2_0(\cX\times\oi)$. 
Then, by \eqref{Bjw1} and \refL{LXX},
\begin{align}\label{Bjw4}
  T_\hF(\psi_1+\psi_2)(x_1,x_2,t)
=T_\hf \psi_1(x_1,t) - T_\hf\psi_2(x_1,t)
-T_\hf \psi_1(x_2,t) + T_\hf\psi_2(x_2,t),
\end{align}
and it follows 
that
$\psi_1+\psi_2$ is an eigenfunction with eigenvalue $\gl\neq0$ if and only if
\begin{align}\label{Bjw5}
  \begin{cases}
T_\hf \psi_1  - T_\hf\psi_2 = \gl\psi_1,
\\
-T_\hf\psi_1 +T_\hf \psi_2 = \gl\psi_2. 
  \end{cases}
\end{align}
These equations imply $\gl(\psi_1+\psi_2)=0$, and thus
$\psi_2=-\psi_1$; furthermore in this case the system simplifies to
$2T_\hf\psi_1=\gl\psi_1$. 
Hence, 
if the nonzero eigenvalues of $T_\hf$ are 
$(\gl_r)_1^R$, 
then the nonzero eigenvalues of $T_\hF$ are $(2\gl_r)_1^R$.

Consequently, \eqref{BUU} and \refT{TU}\ref{TU2} applied to $F$ show that 
\begin{align}\label{Bemm3}
  \frac{1}{n} \Umm_{2n}(f)=\frac{1}{n} U_n(F) +\oLL(1)
\dto \sumrR \gl_r(\zeta_r^2-1)
.\end{align}
This shows that 
$\frac{1}{2n} \Umm_{2n}(f)$ has the limit $W$ in \eqref{tb2}.
In other words, \eqref{tb2} holds for even $n$, 
and thus in the general case.

The first equality in \eqref{Bemms} follows from \eqref{tb2};
the second follows from \eqref{lts2} applied to $\hfiii$ 
together with \eqref{ff} applied to $f_{12}$.

\pfitemx{\ref{TB2s} and \ref{TB2a}}
These follow from \ref{TB2} as in the proof of \refT{TU}.
Alternatively, we may note that \ref{TB2} shows that $\frac{1}{n}\Umm_n(f)$
has the same limit (in distribution) as $\frac1n U_n(f_{12})$, and thus 
\ref{TB2s} and \ref{TB2a} follow directly from the corresponding parts
\ref{TU2s} and \ref{TU2a} in \refT{TU}.

\pfitemref{TBv}
When $f_{12}=0$ and $\mu=0$, \eqref{Baida} simplifies to
\begin{align}\label{Baid6}
\Umm_n(f)&
= 
- \sum_{k=1}^{n/2}f_1(X_{2k-1})  
- \sum_{k=1}^{n/2}f_2(X_{2k})
\end{align}
when $n$ is even, and
\begin{align}\label{Baid7}
\Umm_n(f)&
= 
- \sum_{k=1}^{(n-1)/2}\bigpar{f_1(X_{2k}) +f_2(X_{2k})}
\end{align}
when $n$ is odd. 
The summands in \eqref{Baid6}--\eqref{Baid7} are independent, and
\eqref{tbv01}--\eqref{tbv02} follow directly;
furthermore, 
\eqref{tbv91}--\eqref{tbv92} follow from the central limit theorem.
\end{proof}

\begin{remark}\label{RB}
  As noted above,
there is for the bi-alternating \Ustat{} $\Umm_n$
no case similar to the nondegenerate cases in \refTs{TU} and
\ref{TC} with a variance of order $n^3$.
In fact, apart from \eqref{tb00} and \ref{TBv},
we can summarize \refT{TB} by saying that
(as noted in the proof above),
$\frac{1}{n}\Umm_n(f-\mu)$ 
has the same asymptotic distribution as $\frac1n U_n(f_{12})$.
However, these variables are not the same for finite $n$; in fact, they
are asymptotically uncorrelated, as is easily seen using \refL{LH}.
Moreover, they converge to two \emph{independent} copies of the same $W$;
this may be seen by adapting the method in the proof to show that for any
constants $s,t\in\bbR$,
\begin{align}
  s\cdot\frac{1}{n}\Umm_n(f-\mu)
+t\cdot\frac1n U_n(f_{12})
\dto s\sumrR \tfrac12 \gl_r(\zeta_r^2-1) + t\sumrR \tfrac12 \gl_r(\tzeta_r^2-1),
\end{align}
where $\zeta_r,\tzeta_r$ are independent standard normal variables.
We omit the details.
\end{remark}

The following connection with $\Uc_n$ was noted (and used) by \cite{writhe}
(in a special case).

\begin{proposition}\label{PC=B}
  If $f$ is antisymmetric, and $n$ is odd, then 
  \begin{align}\label{bc1}
    \Uc_{n}(f) \eqd \Umm_{n}(f).
  \end{align}  
\end{proposition}

\begin{proof}
  More precisely, we show that 
  \begin{align}\label{bc2}
    \Uc_n(f;X_2,X_4,\dots,X_{2n}) = \Umm_n(X_1,X_2,\dots,X_n),
  \end{align}
where the indices are interpreted modulo $ n$ as in \refS{SUc};
then \eqref{bc1} follows since $(X_i)_1^n$ are \iid
Note that since $n$ is odd, $i\mapsto2i$ is a bijection of the index set
$\bbZ_n$ onto itself; hence the \lhs{} of \eqref{bc2} contains all variables
$X_1,\dots,X_n$, but in different order.

It follows from the definitions, and the assumption that $f$ is
antisymmetric, that both sides of  \eqref{bc2} are sums containing $\binom
n2$ term of the type $\pm f(X_i,X_j)$, one for each unordered pair \set{i,j}
with $i\neq j$. We only have to verify that the signs agree.
On the \lhs, we have one term $f(X_{2i},X_{2i+2j})$ for every $i\in[n]$
and $j\in[(n-1)/2]$, where $[n]:=\setn$.
Letting  $k\equiv 2i\pmod n$ and $l\equiv2(i+j)\pmod n$ be the
representatives with $k,l\in[n]$, then 
either $k<l$ and $l-k=2j$ is even, or $l<k$ and $k-l=n-2j$ is odd;
conversely, every such pair $(k,l)$ corresponds to a unique pair
$(i,j)\in[n]\times[(n-1)/2]$. 
Since $f$ is antisymmetric, a term $f(X_k,X_l)$ with $l<k$ 
equals $-f(X_l,X_k)$, and thus we see that
the \lhs{} of \eqref{bc1} contains $f(X_k,X_l)$ for $k<l$ with $k-l$ even,
and (interchanging $k$ and $l$),
$-f(X_k,X_l)$ for $k<l$ with $k-l$ odd; this is the same as $\Umm_n$ in
\eqref{U--}. 
\end{proof}

\begin{remark}\label{Reven}
\refP{PC=B} does not hold for even $n$, simply because $\Uc_n$ and $\Umm_n$
then are sums of different numbers of terms and thus, even when 
$f=f_{12}$, they have in general different variances
(using \refL{LH}). 
However, 
for antisymmetric $f$,
\eqref{bc1} holds approximatively with an error $\OLL(n\qq)$ 
also for even $n$ 
as a consequence of \eqref{aid3} and \eqref{Baid3}.
Note also that \eqref{bc1} fails in general for symmetric $f$, even if
$f=f_{12}$; for an example let $f(x,y):=(x-p)(y-p)$ and $X\in\Be(p)$ with
$\frac12<p<1$.  
\end{remark}

\section{Singly alternating \Ustat{s}}\label{SA}
We turn to $\Ump_n$ and $\Upm_n$ in \eqref{U-+}--\eqref{U+-}.
We note first that by arguing as in \eqref{aid3} and \eqref{Baid3}, we see
that
\begin{align}\label{j+-}
  \Ump_n(f)&
=\sum_{1\le i<j\le n}(-1)^{i+1}f(X_i,X_j)
=\sum_{2\le i<j\le n+1}(-1)^{i+1}f(X_i,X_j)+\OLL(n\qq)
\notag\\&
\eqd-\Ump_n(f)+\OLL(n\qq).  
\end{align}
The error term will be negligible in most of our asymptotic results below;
in particular, \eqref{j+-} implies that
any distributional limit found for $n^{-3/2}\Ump_n$ or $n\qw\Ump_n$ has to
be symmetric.
The same holds for $\Upm_n$ by the same argument, or by \eqref{fi2}.
(Note that $U_n$, $\Uc_n$, and $\Umm_n$ can have asymmetric asymptotic
distributions, see \refE{E1}.)
 
\begin{theorem}\label{TA}
With notations and assumptions 
as in \refS{Sprel},
the following holds.
\begin{romenumerate}
\item \label{TA0}
We have
\begin{align}\label{ta00}
  \E [\Upm_n(f)] = (-1)^n\lrflfrac n2\mu
.\end{align}

\item \label{TA1}
As \ntoo,
\begin{align}\label{Agss1}
n^{-3}\Var[\Upm_n(f)]\to
  \gss:=\tfrac{1}{3}\E [f_2(\XX)^2]
\end{align}
and 
\begin{align}\label{ta1}
  n^{-3/2}\lrpar{\Upm_n(f)-\E[\Upm_n(f)]} \dto N(0,\gss).
\end{align}

\item\label{TA2} 
If\/ $f_2(\XX)=0$ a.s., and thus $\gss=0$, then
\begin{align}
  \label{ta20}
  \Var[\Upm_n(f)]
=\tfrac{1}{2}n^2\Var [f_{12}(X_1,X_2)] + O(n).
\end{align}
Moreover, there exists
a finite or infinite sequence of real numbers $(\gl_r)_1^R$ such that
\begin{align}\label{ta2}
n\qw\lrpar{\Upm_n(f)-\E[\Upm_n(f)]} \dto 
  W:=\sumrR\tfrac12\gl_r(\zeta_r^2-1),
\end{align}
where $(\zeta_r)_1^R$ are independent standard normal variables.
The coefficients
$(\gl_r)_1^R$ in \eqref{ta2} are the nonzero eigenvalues (with multiplicities) 
of the self-adjoint  operator 
on $(L^2(\cX\times\oi,\nu\times\Leb))^2$
given in block form by
\begin{align}\label{ta21}
\tfrac12 \matrixx{-T_\hfiii&T_\czfiii\\-T_\czfiii&T_\hfiii}.
\end{align}
We have
\begin{align}\label{ta22}
\Var W=
\tfrac12\sumrR\gl_r^2 = \tfrac12\Var[f_{12}(X_1,X_2)]
<\infty.
\end{align}

\item \label{TA2s}
In the special case of \ref{TA2} where
furthermore
$f$ is symmetric, 
then also
\begin{align}\label{ta2s}
n\qw\lrpar{\Upm_n(f)-\E[\Upm_n(f)]} \dto 
W:=\sumrR\gl_r\xeta_r,
\end{align}
where $(\xeta_r)_1^R$ are \iid{} with the distribution \eqref{xeta},
and the coefficients
$(\gl_r)_1^R$  are the nonzero eigenvalues (with multiplicities) 
of the integral operator $T_f$ on $L^2(\cX,\nu)$.
We have
\begin{align}\label{ta2s2}
  \Var W = \tfrac12\sumrR\gl_r^2 = \tfrac12\Var[f_{12}(X_1,X_2)].
\end{align}

\item \label{TA2a}
In the special case of \ref{TA2} where
furthermore
$f$ is antisymmetric, 
then also
\begin{align}
  \label{ta2a}
  n\qw{\Upm_n(f)} \dto 
W:=\sumqQp\gla_q\eta_q,
\end{align}
where 
$(\eta_q)_1^\Qp$ are independent random variables with the stochastic area
distribution \eqref{area}, and
the coefficients 
$(\gla_q)_1^\Qp$  are the positive numbers such that the imaginary
number $\ii\gla_q$ is an eigenvalue of the 
anti-self-adjoint
operator $T_f$ on $\LLC(\cX,\nu)$.
We have
\begin{align}\label{ta2a2}
  \Var W = \sumqQp(\gla_q)^2 = \tfrac12\Var[f_{12}(X_1,X_2)].
\end{align}

\item \label{TAv}
If $f_2=f_{12}=0$, then
\begin{align}\label{tav01}
  \Var[\Upm_{n}(f)]=n\gss_1+O(1)
\end{align}
where
\begin{align}\label{tav3}
  \gss_1&:=\tfrac12\Var[f_1(X)]
.\end{align}
Furthermore,
\begin{align}\label{tav91}
  n\qqw\bigpar{\Upm_{n}(f)-\E [\Upm_{n}(f)]}&\dto N(0,\gss_1).
\end{align}
\end{romenumerate}
\end{theorem}

\begin{proof}
We follow the proofs of \refTs{TC} and \ref{TB}, again with some
differences;
we omit some details that are the same as above.
First, \ref{TA0} 
follows immediately from \eqref{U+-} and
\begin{align}
  \sum_{1\le i<j\le n} (-1)^{j}
=
  \sumin \sum_{j=i+1}^{n}(-1)^j
=  \sumin (-1)^{n}{\indic{\text{$n-i$ is odd}}}
=(-1)^n\lrfloor{\frac n2}.
\end{align}

We substitute the decomposition \eqref{a2} into the definition \eqref{U+-},
and obtain by simple calculations
\begin{align}\label{Aaida}
\Upm_n(f)&
= \E [\Upm_n(f)]
+ (-1)^n\sumin \indic{\text{$n-i$ is odd}}f_1(X_i)  
+ \sumjn (-1)^j(j-1)f_2(X_j) 
\notag\\&\qquad{}
+ 
\sum_{1\le i< j \le n}(-1)^{j} f_{12}(X_i,X_{j}).
\end{align}
The first sum in \eqref{Aaida} has variance $O(n)$, the second has variance
$O(n^3)$ and the final (double) sum has variance $O(n^2)$.

As in the other proofs, we replace $f$ by $f-\mu$. 
In the rest of the proof we thus may assume that $\mu=0$, and thus 
$\E \Upm_n=0$.

\pfitemref{TA1}
The first and third sums in \eqref{Aaida} can be  ignored, and the remaining
second sum is a sum of independent variables. Hence, a simple calculation
yields \eqref{Agss1}, and \eqref{ta1} follows by the central limit theorem.   

\pfitemref{TA2}
We obtain \eqref{ta20} from \eqref{Aaida}.
For \eqref{ta2},
it follows from \eqref{Aaida} that
it suffices to consider $\Upm(f_{12})$, and thus we may, without loss of
generality, for simplicity assume $f=f_{12}$.
Again it will be convenient to consider even $n$, 
and we note that (since $f=f_{12}$)
\eqref{U+-} implies
\begin{align}\label{Aaid3}
  \Upm_{2n+1}(f) = \Upm_{2n}(f) 
- \sum_{i=1}^{2n} f(X_i,X_{2n+1})
=\Upm_{2n}(f)+\OLL(n\qq)
.\end{align}
Hence it suffices to prove \eqref{ta2} for even $n$.
We consider thus $\Upm_{2n}(f)$, where we assume $n\ge2$.
We use again the definition \eqref{BtX} of $\tX$, so that
$(\tX_i)\xoon$ is an \iid{} sequence of random variables in $\cX^2$.
We now define the function $F$ on $\cX^4=\cX^2\times\cX^2$ by
\begin{align}\label{AF}
  F\bigpar{(x_1,x_2),(y_1,y_2)}
:=-f(x_1,y_1)+f(x_1,y_2)-f(x_2,y_1)+f(x_2,y_2),
\end{align}
and
it follows from the definitions \eqref{U+-} and \eqref{U2} that
\begin{align}\label{AUU}
  \Upm_{2n}(f)&= U_n(F;\tX_1,\dots,\tX_n)
+\sumin f(X_{2i-1},X_{2i})
\notag\\&=
U_n(F;\tX_1,\dots,\tX_n)+\OLL(n\qq)
.\end{align}
Consequently, we may use \refT{TU} applied to $F$ and $(\tX_i)$;
we have again $F_\emptyset=F_1=F_2=0$, so \refT{TU}\ref{TU2} applies.

It remains to find the eigenvalues $\gl_r$ of the integral operator
$T_{\hF}$  on $L^2(\cX^2\times\oi)$ 
which has  kernel,
by \eqref{AF} and recalling both \eqref{hff} and \eqref{czf},
\begin{align}\label{Ajw1}
&\hF\bigpar{(x_1,x_2,t),(y_1,y_2,u)} 
= \bigpar{-f(x_1,y_1)+f(x_1,y_2)-f(x_2,y_1)+f(x_2,y_2)}\indic{t<u}
\notag\\&\qquad
+
\bigpar{-f(y_1,x_1)+f(y_1,x_2)-f(y_2,x_1)+f(y_2,x_2)}\indic{t>u}
\notag\\&
=-\hf\bigpar{(x_1,t),(y_1,u)}
+
\czf\bigpar{(x_1,t),(y_2,u)}
-
\czf\bigpar{(x_2,t),(y_1,u)}
+
\hf\bigpar{(x_2,t),(y_2,u)}
.\end{align}
It follows again from \refL{LXX}
that $T_\hF$ maps $L^2(\cX\times\cX\times\oi)$ into $\MM$, and thus all
eigenfunctions for a nonzero eigenvalue belong to $\MM$. 

Let $\psi_1+\psi_2\in \MM$, with $\psi_j\in M_j$, and let again
$\psi_j$ denote also 
the corresponding function in $L^2_0(\cX\times\oi)$. 
Then, by \eqref{Ajw1} and \refL{LXX}, 
\begin{align}\label{Ajw4}
  T_\hF(\psi_1+\psi_2)(x_1,x_2,t)
=-T_\hf \psi_1(x_1,t) + T_\czf\psi_2(x_1,t)
-T_\czf \psi_1(x_2,t) + T_\hf\psi_2(x_2,t),
\end{align}
and it follows 
that
$\psi_1+\psi_2$ is an eigenfunction with eigenvalue $\gl\neq0$ if and only if
\begin{align}\label{Ajw5}
  \begin{cases}
-T_\hf \psi_1  + T_\czf\psi_2 = \gl\psi_1,
\\
-T_\czf\psi_1 +T_\hf \psi_2 = \gl\psi_2. 
  \end{cases}
\end{align}
Hence the nonzero eigenvalues of $T_\hF$ are the nonzero eigenvalues of
the operator
$\smatrixx{-T_\hf&T_\czf\\-T_\czf&T_\hf}$
on $L^2_0(\hcX)\times L^2_0(\hcX)$. These are the same as the nonzero
eigenvalues on
$L^2(\hcX)\times L^2(\hcX)$, since both $T_\hf$ and $T_\czf$ map 
$L^2(\hcX)$ into $L^2_0(\hcX)$. (Cf.\ \refL{LXX}.)
We denote these eigenvalues by $2\gl_r$, as in the statement, 
and obtain by \eqref{AUU} and \eqref{tu2}
\begin{align}\label{mw2}
  n\qw\Upm_{2n}(f)
=n\qw U_n(F;\tX_1,\dots,\tX_n)+\op(1)
\dto \sumrR \gl_r(\zeta_r^2-1),
\end{align}
which  proves \eqref{ta2} for even $n$, 
and thus, by  \eqref{Aaid3}, in general. 

Regard $\setiii$ as a measure space with mass 1 at each of the two
points.
Then, the operator in \eqref{ta21} can be regarded as the operator $T_G$
on $L^2(\cX\times\oi\times\setiii)$ with kernel $G$ given by the block form
\begin{align}\label{mag1}
G:=\tfrac12\matrixx{-\hf&\czf\\-\czf&\hf}
.\end{align}
It follows that, using \eqref{ff} and its counterpart for $\czf$,
\begin{align}\label{mag2}
  \int_{(\hcX\times\setiii)^2}|G|^2 
= \frac{2}{4}\int_{\hcX^2}|\hf|^2 +  \frac{2}{4}\int_{\hcX^2}|\czf|^2
=\int_{\cX^2}|f|^2.
\end{align}
Hence, \eqref{ta22} follows from \eqref{ta2} and \eqref{lts2}.

\pfitemref{TA2s}
Since $f$ is symmetric and $f_2(X)=0$ a.s., we also have $f_1(X)=0$ a.s.,
and thus (since we assume $\mu=0$) $f=f_{12}$. 
Furthermore, the definitions \eqref{hff} and \eqref{czf} yield
\begin{align}\label{mag3}
  \hf\bigpar{(x,t),(y,u)} &= f(x,y),
\\\label{mag4}
  \czf\bigpar{(x,t),(y,u)} &= f(x,y)\sgn(u-t).
\end{align}
Hence, if we define the symmetric
function $\Hs$ on $(\oi\times\setiii)^2$ by the block
form
\begin{align}\label{mag5}
  \Hs\bigpar{(t,\ga),(u,\gb)}:=\frac12\matrixx{-1&\sgn(u-t)\\-\sgn(u-t)&1},  
\qquad t,u\in\oi;
\ga,\gb\in\setiii
,\end{align}
then
we can regard the kernel $G$ in \eqref{mag1} as the tensor product
$f\tensor \Hs$ in the natural way, and thus
\begin{align}\label{mag6}
  T_G = T_f\tensor T_\Hs,
\end{align}
where both $T_f$ and $T_\Hs$ are self-adjoint.
It follows,
by the same argument as in the proof of \refL{Lanti} in a similar case,
that if $T_f$ has the nonzero eigenvalues $\set{\gl_r:r\in\cR}$
and $\Hs$ has the nonzero eigenvalues $\set{\rho_s:s\in\cS}$, then
$T_G$ has the  nonzero eigenvalues $\set{\gl_r\rho_s:r\in\cR,s\in\cS}$.
The eigenvalues $\rho_s$ are given by \refL{LAs} below.
Hence, the limit in \eqref{ta2} is
\begin{align}\label{magg}
  W&
=\sumrR\sumkk\tfrac12\gl_r\frac{2}{(2k-1)\pi}(\zeta_{r,k}^2-1)
=\sumrR\gl_r\sumkk\frac{1}{(2k-1)\pi}(\zeta_{r,k}^2-1)
\notag\\&
=:\sumrR\gl_r\xeta_r,
\end{align}
where $\xeta_r$ are \iid{} and by \refL{Larea} have the distribution
\eqref{xeta}. 

Finally, \eqref{ta2s2} follows from \eqref{ta2s} and \eqref{xetavar} together
with \eqref{ta22}.

\pfitemref{TA2a}
Since $f$ is antisymmetric, we again see that $f_2(X)=0$ a.s.\ implies 
$f_1(X)=0$ a.s., and thus  $f=f_{12}$. 
The definitions \eqref{hff} and \eqref{czf} now yield
\begin{align}\label{mag3a}
  \hf\bigpar{(x,t),(y,u)} &= f(x,y)\sgn(u-t).
\\\label{mag4a}
  \czf\bigpar{(x,t),(y,u)} &= f(x,y).
\end{align}
Hence, if we now define the antisymmetric
function $\Ha$ on $(\oi\times\setiii)^2$ by the block form
\begin{align}\label{mag5a}
  \Ha\bigpar{(t,\ga),(u,\gb)}:=\frac12\matrixx{-\sgn(u-t)&1\\-1 &\sgn(u-t)}, 
\qquad t,u\in\oi;
\ga,\gb\in\setiii
\end{align}
then
we have again \eqref{mag6},
where now both $T_f$ and $T_\Ha$ are anti-self-adjoint, and thus have imaginary
eigenvalues. 
It follows
that if $T_f$ has the nonzero eigenvalues $\set{\ii\gl'_q:q\in\cQ}$
and $\Ha$ has the nonzero eigenvalues $\set{\ii\rho_s:s\in\cS}$, then
$T_G$ has the  nonzero eigenvalues $\set{-\gl'_q\rho_s:q\in\cQ,s\in\cS}$.
The eigenvalues $\ii\rho_s$ are given by \refL{LAa} below.
Hence, 
the limit in
\eqref{ta2} is, cf.\ \eqref{magg},
\begin{align}\label{magga}
  W&
=\sum_{q\in\cQ}\sumkk\gl'_q\frac{-1}{(2k-1)\pi}(\zeta_{r,k}^2-1)
=:\sum_{q\in\cQ}\gl'_q\xeta_q,
\end{align}
where $(\xeta_q)_{q\in\cQ}$ are \iid{} and by \refL{Larea} have the distribution
\eqref{xeta}. 
Furthermore, since $f$ is real, the nonzero eigenvalues of $T_f$
are symmetric with respect to the real axis, and are thus 
$(\pm\ii\gla_q)_{q=1}^\Qp$.
Hence, we can rewrite \eqref{magga} as
\begin{align}\label{maggb}
  W&
=\sumqQp(\gla_q\xeta_q-\gla_q\xeta'_q)
=\sumqQp\gla_q(\xeta_q-\xeta'_q)
=\sumqQp\gla_q\eta_q,
\end{align}
where all $\xeta_q$ and $\xeta'_q$ are \iid{} with the distribution
\eqref{xeta}, and thus
$\eta_q:=\xeta_q-\xeta'_q$ are \iid{} with the distribution \eqref{area} by
\eqref{xeta}, see also \refL{Larea}.
 
Finally, \eqref{ta2a2} follows from \eqref{ta2a} and \eqref{etavar} together
with \eqref{ta22}.

\pfitemref{TAv}
Follows from  \eqref{Aaida} and the central limit theorem.
\end{proof}

\begin{remark}\label{R+-}
As noted after \eqref{j+-}, the limits in distribution
in \refT{TA} have to be symmetric  random variables.
This is obvious in \ref{TA1} and \ref{TA2s}--\ref{TAv}, but in \ref{TA2}, it
implies that the set of eigenvalues $(\gl_r)_1^R$ has to be symmetric, i.e., 
$(\gl_r)_1^R$ equals $(-\gl_r)_1^R$ up to order.
This can also be seen from \eqref{ta21}: the \mpp{} bijection
$(\hx,\hy)\mapsto(\hy,\hx)$ of $\hX\times\hX$ onto itself induces a unitary
equivalence of the operator \eqref{ta21} with its negative.
As a consequence, the limit in \eqref{ta2} can also be written
\begin{align}
W= \sum_{\gl_r>0}\tfrac12\gl_r(\zeta_r^2-\tzeta^2_r),
\end{align}
where $\zeta_r,\tzeta_r$ are independent standard normal variables.
\end{remark}

\begin{theorem}\label{TA-+}
With notations and assumptions 
as in \refS{Sprel},
the following holds.
\begin{romenumerate}
\item \label{TA-+0}
We have
  \begin{align}\label{ta-+00}
  \E \Ump_n(f) = \lrflfrac n2\mu
.\end{align}

\item \label{TA-+1}
All conclusions of \refT{TA}\ref{TA1}--\ref{TAv} hold also 
for $\Ump(f)$ instead of $\Upm(f)$, provided $f_2$ is replaced by $f_1$ and
conversely. 
\end{romenumerate}
\end{theorem}
\begin{proof}
This follows from \eqref{fi2} and \refT{TA} (applied to $\xf$)
together with the following observations.
First,
$\xf_1=f_2$, $\xf_2=f_1$, and $(\xf)_{12}=\xfiii$.
Secondly,
the factor $(-1)^n$ in \eqref{fi2} does not matter in \ref{TA-+1},
since the limits are symmetric, see \refR{R+-}.
Thirdly,
the operator \eqref{ta21} and its counterpart for $\xf_{12}$ 
are unitarily equivalent and thus have the same eigenvalues, 
with the unitary equivalence induced by the \mpp{} bijection 
$\bigpar{(x,t),(y,u)}\mapsto\bigpar{(y,1-u),(x,1-t)}$
of $\hcX^2$ onto itself, 
since
\begin{align}
  \widehat{\xf_{12}}\bigpar{(x,t),(y,u)}
=
  \widehat{f_{12}}\bigpar{(y,1-u),(x,1-t)}
\end{align}
and similarly for $\widecheck{\xf_{12}}$ and $\widecheck{f_{12}}$.
\end{proof}

\begin{lemma}\label{LAs}
  Let $\Hs$ be the symmetric function on $(\oi\times\setiii)^2$ given by
  \eqref{mag5}.
Then the eigenvalues of $T_\Hs$, all simple, are
\begin{align}\label{las}
\pm \frac{2}{(2k-1)\pi}, \qquad k=1,2,3,\dots
.\end{align}
\end{lemma}

As a sanity check we note that if the eigenvalues in \eqref{las} are
enumerated $(\gl_r)_1^\infty$, then
\begin{align}
  \sum_{r}\gl_r^2 =\frac{4}{\pi^2}\cdot2\sumk\frac{1}{(2k-1)^2}=1
=\int_{(\oi\times\setiii)^2}|\Hs|^2
\end{align}
since $|\Hs|=\frac12$ and $\oi\times\setiii$ has measure 2; this agrees with
\eqref{lts2}. 

\begin{proof}
  An eigenfunction of $T_\Hs$, with eigenvalue $\gl$, 
is a pair $(\gf_1,\gf_2)$ of functions on $\oi$
  such that
  \begin{align}\label{em01}
2\gl\gf_1(t) &
=- \intoi \gf_1(u)\dd u-\int_0^t\gf_2(u)\dd u+\int_t^1 \gf_2(u)\dd u,
\\\label{em02}
2\gl\gf_2(t) &= \int_0^t\gf_1(u)\dd u-\int_t^1\gf_1(u)\dd u+\intoi \gf_2(u)\dd u.
  \end{align}
Suppose $\gl\neq0$.
It then follows, as in the proof of \refL{L3}, that $\gf_1$ and $\gf_2$ are
continuously differentiable, and differentiation yields
  \begin{align}\label{em1}
2\gl\gf'_1(t)& =- 2\gf_2(t),
\\\label{em2}
2\gl\gf'_2(t) &= 2\gf_1(t).
  \end{align}
Let $\go:=1/\gl$. Then the system \eqref{em1}--\eqref{em2} becomes
\begin{align}\label{em3}
\gf_1'&=-\go\gf_2,\\
\gf_2'&=\go\gf_1.\label{em4} 
\end{align}
It follows that $\gf_1''=-\go^2\gf_1$, and thus, for some constants $a$ and
$b$, using also \eqref{em3} again, 
\begin{align} \label{em5}
  \gf_1(t)&=a\cos(\go t)+b\sin(\go t),
\\\label{em6}
  \gf_2(t)&=a\sin(\go t)-b\cos(\go t),
\end{align}
Furthermore, taking $t=0$ in \eqref{em01} and \eqref{em02}, and integrating
using \eqref{em1}--\eqref{em2},
\begin{align}\label{em7}
2\gl\gf_1(0)
= 2\gl\gf_2(0)&
=-\intoi \gf_1(t)\dd t +\intoi\gf_2(t)\dd t
\notag\\&
=-{\gl}\bigpar{\gf_2(1)-\gf_2(0)}
 -{\gl}\bigpar{\gf_1(1)-\gf_1(0)},
\end{align}
which implies 
$\gf_1(0)=\gf_2(0)$ and thus $a=-b$, and then $\gf_1(1)+\gf_2(1)=0$ and thus,
by adding \eqref{em5} and \eqref{em6},
\begin{align}\label{em8}
  \cos(\go)=0.
\end{align}
Hence, for some $k\in\bbZ$,
\begin{align}
  \go = (k+\tfrac12)\pi.
\end{align}
Conversely, it follows that for each such $\go$, we obtain an eigenfunction
with eigenvalue $\gl=1/\go$ by \eqref{em5}--\eqref{em6} with $b=-a$. 
This shows that eigenvalues are \eqref{las};
we see also that these eigenvalues are simple.

For completeness we note that $0$ is not an eigenvalue, since 
\eqref{em01}--\eqref{em02} with  $\gl=0$ imply that the \rhs{s} do not
depend on $t$, and thus $\gf_1(t)=\gf_2(t)=0$ a.e.
\end{proof}

\begin{lemma}\label{LAa}
  Let $\Ha$ be the antisymmetric function on $(\oi\times\setiii)^2$ given by
  \eqref{mag5a}.
Then the eigenvalues of $T_\Ha$, all simple, are
\begin{align}\label{laa}
\pm \frac{2\ii}{(2k-1)\pi}, \qquad k=1,2,3,\dots
.\end{align}
\end{lemma}

\begin{proof}
We argue as in the proof of \refL{LAs}.
  An eigenfunction of $T_\Ha$, with eigenvalue $\gl$, 
is now a pair $(\gf_1,\gf_2)$ of functions on $\oi$
  such that
  \begin{align}\label{el01}
2\gl\gf_1(t) &
=
\int_0^t\gf_1(u)\dd u-\int_t^1\gf_1(u)\dd u+\intoi \gf_2(u)\dd u,
\\\label{el02}
2\gl\gf_2(t) &= 
- \intoi \gf_1(u)\dd u-\int_0^t\gf_2(u)\dd u+\int_t^1 \gf_2(u)\dd u.
  \end{align}
Suppose $\gl\neq0$.
It then follows, that $\gf_1$ and $\gf_2$ are
continuously differentiable, and differentiation yields
  \begin{align}\label{el1}
2\gl\gf'_1(t)& =2\gf_1(t),
\\\label{el2}
2\gl\gf'_2(t) &= -2\gf_2(t).
  \end{align}
Let $\go:=-\ii/\gl$. Then the system \eqref{el1}--\eqref{el2} becomes
\begin{align}\label{el3}
\gf_1'&=\ii\go\gf_1,\\
\gf_2'&=-\ii\go\gf_2.\label{el4} 
\end{align}
Thus, for some constants $a$ and
$b$,
\begin{align} \label{el5}
  \gf_1(t)&=a e^{\ii\go t},
\\\label{el6}
  \gf_2(t)&=be^{-\ii\go t}.
\end{align}
Furthermore, taking $t=0$ in \eqref{el01} and \eqref{el02}, and integrating
using \eqref{el1}--\eqref{el2},
\begin{align}\label{el7}
2\gl\gf_1(0)
= 2\gl\gf_2(0)&
=-\intoi \gf_1(t)\dd t +\intoi\gf_2(t)\dd t
\notag\\&
=-{\gl}\bigpar{\gf_1(1)-\gf_1(0)}
 -{\gl}\bigpar{\gf_2(1)-\gf_2(0)},
\end{align}
which implies first $\gf_1(0)=\gf_2(0)$ and thus  $a=b$,
 and then $\gf_1(1)+\gf_2(1)=0$ and thus,
by adding \eqref{el5} and \eqref{el6},
\begin{align}\label{el8}
2\cos(\go)=e^{\ii\go}+e^{-\ii\go}=0.
\end{align}
Hence, for some $k\in\bbZ$,
\begin{align}
  \go = (k+\tfrac12)\pi.
\end{align}
Conversely, it follows that for each such $\go$, we obtain an eigenfunction
with eigenvalue $\gl=-\ii/\go$ by \eqref{el5}--\eqref{el6} with $b=a$. 
This shows that eigenvalues are \eqref{laa};
we see also that these eigenvalues are simple.

By the same argument as in the proof of \refL{LAs},
$0$ is not an eigenvalue.
\end{proof}

\begin{remark}\label{Runitary}
  \refLs{LAs} and \ref{LAa} show that the self-adjoint operators $T_\Hs$ and
  $\ii T_\Ha$ have the same eigenvalues, and thus are unitarily equivalent.
Operators with the same eigenvalues appear also in \refE{ESJ22}
and \refL{L3} below.
A unitary equivalence between any two of these
is given by mapping eigenfunctions to eigenfunctions
with the same eigenvalue but, in spite of the simple explicit forms of the
eigenfunctions found in the proofs (and the great similarties between the
proofs above and below),
we do not see a simple explicit form of
the unitary equivalences except for the case of \refL{LAs} and \refE{ESJ22}.
\end{remark}

\section{A short summary}\label{SSummary}
Comparing \refTs{TU}, \ref{TC}, \ref{TB}, \ref{TA}, and \ref{TA-+}
 we see  strong similarities but also differences.
The nondegenerate cases are similar, with variances of
order $n^3$ and  normal limits; the proofs show that in these case, the
dominating terms are linear combinations of $f_1(X_i)$ and $f_2(X_i)$,
but the details differ because the linear combinations that appear are different
for the different \Ustat{s}. For $\Ump$ and $\Upm$ this is due to partial
cancellations caused by the alternating signs, and for $\Umm$ this
cancellation is (almost) complete so that the nondegenerate case does not
occur at all.
The asymptotic variances in \eqref{gss1}, \eqref{gss1c}, and \eqref{Agss1}
are in general different, but note that in the special case when $f$ is
antisymmetric, and thus $f_2=-f_1$, 
and further $f_1\neq0$, it follows that
$U_n(f)$ and $\Upm_n(f)$ (and $\Ump_n(f)$)
have the same asymptotic variance and thus the same asymptotic distribution,
while $\Uc_n(f)$ is of the degenerate type, and has the same asymptotic
distribution as $\Umm_n(f)$.

For the degenerate cases the general pattern is again similar, but details
differ in more subtle and nonobvious ways. 
To see this clearer, we collect in the corollaries below the results for
the degenerate case when $f$ is symmetric or antisymmetric; we further
assume for simplicity $f=f_{12}$, 
i.e., $\mu=0=f_1=f_2$, since the degenerate cases
always reduce to this case. Note that the variables $f_{12}(X_i,X_j)$ are
orthogonal by \refL{LH}; hence, 
when $f=f_{12}$,
alternating signs do not change
the variance of the \Ustat{} and do not
cause any cancellation, although they may affect the asymptotic distribution.
We let $\zeta_r\in N(0,1)$, $\eta_q$, and $\xeta_r$ be independent copies of
the variables in \refSS{SS3}.

\begin{corollary}
  \label{Csymm}
Suppose that $f$ is symmetric and that $f=f_{12}$. 
Let the nonzero eigenvalues of $T_f$ be $(\gl_r)_1^R$.
Then
\begin{align}\label{sum1}
  n\qw{U_n(f)} \dto W:=\sumrR\tfrac12\gl_r(\zeta_r^2-1).
\end{align}
The same result holds for $\Uc_n$ and $\Umm_n$.
Furthermore,
\begin{align}\label{sum2}
  n\qw{\Upm_n(f)} \dto W:=\sumrR\gl_r\xeta_r
.\end{align}
The same result holds for $\Ump_n$.
\end{corollary}

The limit distributions in \eqref{sum1} and \eqref{sum2} are different by
\refR{Rsum} (unless $f$ is constant and thus both limits are 0).

\begin{corollary}
  \label{Casymm}
Suppose that $f$ is symmetric and that $f=f_{12}$. 
Let the nonzero
eigenvalues of $T_f$ with positive imaginary part be $(\ii\gla_q)_1^\Qp$.
Then
\begin{align}\label{sum3}
  n\qw{U_n(f)} \dto W:=\sumqQp\gla_q\eta_q.
\end{align}
The same result holds for $\Uc_n$, $\Umm_n$, $\Upm_n$, and $\Ump_n$.
\end{corollary}

Recall from \refR{RB} that the fact that two of our \Ustat{s} have the same
limit distribution does not imply that they converge jointly to the same random
variable. On the other hand, \refR{Rsymm} shows that this happens in the
case of $U_n(f)$ and $\Uc_n(f)$ for symmetric $f$.
See further \refSS{SSjoint}.

Consider now the general case $f=f_{12}$, without any symmetry assumption.
In this case, 
we have seen in \refT{TC}\ref{TC2} 
and \refR{Rdecouple}
that 
the contributions to $\Uc_n(f)$ from the 
symmetric and antisymmetric parts are asymptotically independent;
hence the asymptotic distribution in the general case follows from the special
cases in \refCs{Csymm} and \ref{Casymm}. 
However, as is shown in \refE{ESJ22}, this does not hold for
$U_n$, $\Umm_n$, $\Upm_n$, or $\Ump_n$.

\section{Examples}\label{Sex}

Examples with nondegenerate, and thus normal, limits are straightforward,
so we concentrate on limits of the more complicated degenerate type.
Again,
we let $\zeta\in N(0,1)$, $\eta$, and $\xeta$,
with or without subscripts, be independent copies of
the variables in \refSS{SS3}.

\begin{example}\label{E1}
  Consider first the simple example 
where $f(x,y)=xy$
and
$X$ is real-valued with finite
variance $\gss_X>0$.
Let $\mu_X:=\E X$.
Then the Hoeffding decomposition
\eqref{a2}--\eqref{a6} is given by
$\mu=\mu_X^2$, $f_1(x)=f_2(x)=\mu_X(x-\mu_X)$, and
$f_{12}(x,y)=(x-\mu_X)(y-\mu_X)$. Hence, if $\mu_X\neq0$,
we have the nondegenerate case with variance of order $n^3$ and normal
limits for $U_n$, $\Uc_n$, $\Ump_n$, and $\Upm_n$. Recall that $\Umm_n$
never has this behaviour; in this example 
\refT{TB}\ref{TB2} shows that $\Umm_n(f)$ has the
same asymptotic distribution as $\Umm_n(f_{12})$, which is equivalent to
replacing $X$ by $X-\mu_X$.

Suppose now that $\mu_X=0$. The integral operator $T_f$ is
\begin{align}
T_f g(x) = x\int y\,g(y)\dd \nu(y) = x\E[Xg(X)].  
\end{align}
This operator has a single nonzero eigenvalue $\gss_X$ with eigenvector
$\gf(x)=x$; hence \refT{TU} yields
\begin{align}
  n\qw U_n\dto \tfrac12\gss_X (\zeta^2-1),
\end{align}
where $\zeta\in N(0,1)$.
\refTs{TC} and \ref{TB} yield the same result for $\Uc_n$ and $\Umm_n$
(\cf{} \refRs{Rsymm} and \ref{RB}),
while \refTs{TA}\ref{TA2s} and \ref{TA-+} yield
\begin{align}
  n\qw \Upm_n\dto \gss_X \xeta
\end{align}
and the same for $\Ump_n$, with $\xeta$ as in \eqref{xeta}.
These results also follow from \refC{Csymm}.
\end{example}

\begin{example}\label{ESJ22}
Let $X=(\xi_1,\xi_2)$ be a random vector in $\cX=\bbR^2$, with $\xi_1$ and
$\xi_2$ independent and both having variance 1 and symmetric distributions,
i.e., $\xi_j\eqd-\xi_j$.
(For example, $\xi_1$ and $\xi_2$ may both be $N(0,1)$, or uniformly
distributed on $\pm1$.) 
Let
\begin{align}\label{ewi1}
  f\bigpar{(x_1,x_2),(x'_1,x'_2)}:=x_1x'_2.
\end{align}
Then, writing $X_n=(\xi_{n1},\xi_{n2})$,
\begin{align}\label{ewi2}
  U_n(f)=\sum_{1\le i<j\le n}\xi_{i1}\xi_{j2}
=\sumjn \xi_{j2}\sum_{i=1}^j\xi_{i1}.
\end{align}
Note that $\Upm_n(f)$ is given by the same sum with $\xi_{j2}$ replaced by
$(-1)^j\xi_{j2}$. Since we assume
$(-1)^j\xi_{j2}\eqd\xi_{j2}$, and all $\xi_{ik}$ are independent, it follows
that
$\Upm_n(f)$ has the same distribution as $U_n(f)$.
The same argument applies also to $\Ump_n(f)$, now replacing
$\xi_{i1}$ by $(-1)^i\xi_{i1}$, and to $\Umm_n(f)$ (doing both).
Thus, for any $n\ge1$,
\begin{align}\label{ewi3}
  U_n(f)\eqd\Umm_n(f)\eqd\Upm_n(f)\eqd\Ump_n(f).
\end{align}
It is shown in \cite[p.~83]{SJ22} that
\begin{align}\label{ewi4}
  n\qw U_n(f)\dto \intoi B_1(t)\dd B_2(t)=:\xeta,
\end{align}
where $B_k(t)$ are independent Brownian motions and thus $\xeta$ is as in
\refSS{SS3}.
(This is a consequence of \eqref{ewi2} and Donsker's theorem applied to 
$\sumin \xi_{i1}$ and $\sumin \xi_{i2}$; see \cite{SJ22} for the nontrivial
technical details.) By \eqref{ewi3}, we have the same asymptotic
distribution \eqref{ewi4} for $\Umm_n(f)$, $\Upm_n(f)$, and $\Ump_n(f)$.

We can also obtain this limit from \refT{TU}\ref{TU2} (or
\refT{TB}\ref{TB2}) above; note that $\mu=f_1=f_2=0$ so $f=f_{12}$.
It follows from the definitions \eqref{hff} and \eqref{Tg} that $T_\hf$ is
the integral operator on $L^2(\bbR^2\times\oi)$ given by
\begin{align}\label{ewi5}
  T_\hf \gF(x_1,x_2,t)&
=\int_t^1 \E\bigsqpar{x_1 \xi_2 \gF(\xi_1,\xi_2 ,u)}\dd u
+
\int_0^t \E\bigsqpar{\xi_1x_2 \gF(\xi_1,\xi_2 ,u)}\dd u
\notag\\&
=x_1\int_t^1 \E\bigsqpar{\xi_2 \gF(\xi_1,\xi_2 ,u)}\dd u
+
x_2\int_0^t \E\bigsqpar{\xi_1\gF(\xi_1,\xi_2 ,u)}\dd u.
\end{align}
Consequently, an eigenfunction with a nonzero eigenvalue $\gl$ has to be of
the form $\gF(x_1,x_2,t)=x_1\psi_1(t)+x_2\psi_2(t)$. Substitution in
\eqref{ewi5} then yields
\begin{align}\label{ewi6}
  \gl x_1\psi_1(t)+\gl x_2\psi_2(t)
=x_1\int_t^1\psi_2(u)\dd u
+x_2\int_0^t\psi_1(u)\dd u
\end{align}
and thus
\begin{align}\label{ewi7}
\gl \psi_1(t)&=\int_t^1\psi_2(u)\dd u,
\qquad
 \gl\psi_2(t)=\int_0^t\psi_1(u)\dd u.
\end{align}
Consequently,
\begin{align}\label{ewi8}
\gl \psi'_1(t)&=-\psi_2(t),
\qquad
 \gl\psi'_2(t)=\psi_1(t).
\end{align}
It is easily seen that \eqref{ewi8}, with the boundary values
$\psi_1(1)=\psi_2(0)=0$ given by \eqref{ewi7}, is solved by, 
with $\go:=1/\gl$,
\begin{align}\label{ewi9}
 \psi_1(t)&=C \cos(\go t),
\qquad
\psi_2(t)=C\sin(\go t),
\end{align}
where we must have
\begin{align}\label{ewi10}
  \cos(\go)=0.
\end{align}
Hence $\go=(k+\frac12)\pi$, $k\in\bbZ$.
The nonzero eigenvalues $\gl=1/\go$ are thus
$\set{\frac{2}{(2k+1)\pi}:k\in\bbZ}$,
and \eqref{tu2} yields
\begin{align}\label{ewi11}
  n\qw U_n(f)\dto \sumkk  \frac{1}{(2k+1)\pi}(\zeta^2_k-1),
\end{align}
which by \refL{Larea}  has the same distribution as $\xeta$ in \eqref{xeta},
which proves \eqref{ewi4}.
As noted above, we have the same limit \eqref{ewi4} also for $\Upm_n(f)$ and
$\Ump_n(f)$. In principle, this can be shown as above from \refT{TA}, but
that would require studying the more complicated integral operator \eqref{ta21}.

Consider now the symmetric and antisymmetric parts; for convenience we
consider
\begin{align}\label{ewis}
 2\fs\bigpar{(x_1,x_2),(x'_1,x'_2)}&=x_1x'_2+x_2x'_1,
\\\label{ewia}
 2\fa\bigpar{(x_1,x_2),(x'_1,x'_2)}&=x_1x'_2-x_2x'_1.
\end{align}
Arguing as after \eqref{ewi5} above, we see that it suffices to consider the
subspace of linear functions $\set{ax_1+by_1:a,b\in\bbC}\subset \LLC(\cX,\nu)$.
This subspace is two-dimensional, and it is easy to see, in analogy to
\eqref{ewi5} but simpler, that in this subspace, $T_{2\fs}$ and $T_{2\fa}$
act by the matrices 
$\smatrixx{0&1\\1&0}$
and
$\smatrixx{0&1\\-1&0}$; the nonzero eigenvalues are thus $\pm1$ and
$\pm\ii$, respectively.
Consequently, \refT{TU}\ref{TU2s} and \ref{TU2a} yield
\begin{align}\label{qj1}
  2n\qw U_n(\fs)&
=  n\qw U_n(2\fs)
\dto \tfrac12(\zeta_1^2-\zeta_2^2),
\\\label{qj2}
  2n\qw U_n(\fa)&
=  n\qw U_n(2\fa)
\dto \eta.
\end{align}
This was also shown in \cite[p.~83]{SJ22}, representing the limits as
$\intoi B_1(t)\dd B_2(t)\pm\intoi B_2(t)\dd B_1(t)$ 
in analogy to \eqref{ewi4} above.
See also \eqref{yor}, which implies that if we denote these limits in
\eqref{qj1}--\eqref{qj2} by
$\Ws$ and $\Wa$, then, as noted in \cite{SJ22}, their joint \chf{} is
\begin{align}\label{york}
\E \bigsqpar{\exp \xpar{\ii s\Ws+\ii t\Wa}}    
=\Bigpar{\cosh^2(t)+s^2\frac{\sinh^2(t)}{t^2}}\qqw.
\end{align}
In particular, \eqref{qj1}--\eqref{qj2} hold jointly, 
but the limits $\Ws$ and $\Wa$ are \emph{not} independent.

By \refC{Csymm},
\eqref{qj1} holds also for $\Uc_n(\fs)$ and $\Umm_n(\fs)$,
while, by \eqref{sum2},
\begin{align}\label{qj3}
  2n\qw \Upm_n(\fs)&\dto \xeta_1-\xeta_2\eqd \eta
\end{align}
and the same for $\Ump_n(\fs)$.
Similarly, by  \refC{Casymm}, 
\eqref{qj2} holds also for $\Uc_n(\fa)$, $\Umm_n(\fa)$, $\Upm_n(\fa)$, and
$\Ump_n(\fa)$.
Note that \eqref{qj1} and \eqref{qj3} show that $U_n(\fs)$ and $\Upm_n(\fs)$
have different limits in distribution; in particular, \eqref{ewi3} cannot be
extended to $\fs$.

We have so far ignored $\Uc_n(f)$, but armed with these results for $\fs$ and
$\fa$, we obtain from \eqref{tc2} and \refL{Larea}
\begin{align}\label{qk1}
  n\qw\Uc_n(f) &\dto 
\tfrac14(\zeta_1^2-\zeta_2^2) + \tfrac12\eta
\notag\\&
\eqd
\tfrac14(\zeta_1^2-\zeta_2^2) 
+ \sumk\frac{1}{2(2k-1)\pi}
\bigpar{\zeta^2_{k,1}+\zeta^2_{k,2}-\zeta^2_{k,3}-\zeta^2_{k,4}}
\end{align}
This differs, by \eqref{larea0} and the uniqueness assertion in \refR{Rsum},
from the limit $\xeta$ found in \eqref{ewi3}--\eqref{ewi4} for 
$U_n$, $\Umm_n$, $\Upm_n$, and $\Ump_n$.
(Note that \eqref{qk1} contains some coefficients that are rational, and
some that are rational multiples of $1/\pi$.)
It follows similarly that the decoupling of the contributions from the symmetric
and antisymmetric parts that is seen in \refT{TC}\ref{TC2} is unique to
$\Uc_n$, and does not hold for the other \Ustat{s} considered here.
\end{example}

\begin{remark}\label{RWP}\kolla
The formula \eqref{york} for 
the joint asymptotic distribution
of $U_n(\fs)$ and $U_n(\fa)$
can also be obtained from \refT{TU} applied to $s\fs+t\fa$, 
see \refApp{AYor}.
$\Umm_n(\fs)$ and $\Umm_n(\fa)$ have the
same joint asymptotic distribution, by the same sign-change argument as for
\eqref{ewi3}, but note that this argument does not apply to
$\Upm_n(\fs)$ and $\Upm_n(\fa)$, as is shown by \eqref{qj3}.
\refT{TA} applied to $s\fs+t\fa$ shows that
$\Upm_n(\fs)$ and $\Upm_n(\fa)$ have a joint asymptotic distribution, which
in principle can be found by 
arguments similar to \refApp{AYor} (but for the more complicated operator
\eqref{ta21});  we have not pursued this and leave it as an open problem
\end{remark}

\begin{remark}\label{Rperm}
  In this paper, we generally assume that $X_1,\dots,X_n$ are \iid{} random
  variables. However, the definitions \eqref{U1}--\eqref{Uc} and
  \eqref{U3}--\eqref{U--} make sense for any deterministic or random
  sequence $X_1,\dots,X_n$.
One interesting instance of this is to let $\bgs=(X_1,\dots,X_n)$ be a
permutation of \set{1,\dots,n}. If further $f(x,y):=\indic{x>y}$; then 
$U_n(f;\bgs)$ is the number of inversions in $\bgs$; we will in the
following two examples consider the equivalent (and more symmetric)
\begin{align}\label{fsgn}
f(x,y):=\sgn(x-y)=2\indic{x>y}-1.
\end{align}
We furthermore take $\bgs$ to be a uniformly random permutation in the
symmetric group $\fS_n$.
It is well-known that $\bgs$ can be constructed as the ranks of a sequence
$X_1,\dots,X_n$ of \iid{} random variables with, say, a uniform distribution
on $\oi$. Since $f$ only cares about the order relations, it follows that then
$U_n(f;\bgs)=U_n(f;X_1,\dots,X_n)$,
and similary for the other \Ustat{s}; hence we are back to the case of \iid{}
$X_i$. 
\end{remark}

\begin{example}[writhe]\label{Ewrithe}
\citet{writhe} defines the \emph{writhe} of a permutation $\bgs\in \fS_{2n+1}$
as, in our notation in \refR{Rperm}, $\Uc_{2n+1}(f;\bgs)$, where 
$f(x,y):=\sgn(x-y)$ as in \eqref{fsgn},
and  
studied this in the case of a uniformly random permutation $\bgs\in\fS_{2n+1}$.
This was motivated by the study of a model for random knots; see
\cite{writhe} for details. (Only permutations of odd lengths appear in this
model.) 

The main result of \cite{writhe} finds the asymptotic distribution of the writhe
as \ntoo; this is proved using the method of moments, together with a
lengthy (but interesting) combinatorial calculation of the moments.
(The proof actually uses the equivalence with $\Umm_{2n+1}$ in \refP{PC=B},
which was given in \cite{writhe} for this case; the moment calculations
there are done for $\Umm_{2n+1}$.)

As said in \refR{Rperm}, the distribution of the writhe
equals the distribution of $\Uc_{2n+1}(f;X_1,\dots,X_{2n+1})$ where 
$X_i\in U(0,1)$ are \iid.
Consequently, we may apply \refT{TC}. 
The kernel \eqref{fsgn} is alternating, and thus $\mu=0$.
Furthermore, 
\begin{align}\label{ez1}
  f_1(x)&=\E[\sgn(x-X)] = 2x-1,
\\\label{ez2}
f_2(x)&=-f_1(x)=1-2x,
\\\label{ez12}
f_{12}(x,y)&=\sgn(x-y)-2x+2y,
\end{align}
and we find $\Var[f_{12}(X_1,X_2)]=\frac{1}{3}$.
We apply \refT{TC}\ref{TC2}, with $\fiiis=0$ and $\fiiia=f_{12}$, and it
remains to find the eigenvalues of $T_{\fiiia}=T_{f_{12}}$.
Note that the eigenvalues of $T_f$ are given in \refL{L3}, but here we
consider $T_{f_{12}}$.

Suppose that $\gf$ is an eigenfunction of $T_{f_{12}}$ with eigenvalue
$\ii\gl$. (All eigenvalues are imaginary, since $f_{12}$ is antisymmetric.)
Then
\begin{align}\label{ez3}
  \ii\gl\gf(x) = \int_0^x \gf(y)\dd y  - \int_x^1 \gf(y)\dd y 
-2x\intoi\gf(y)\dd y + 2\intoi y\gf(y)\dd y.
\end{align}
Suppose that $\gl\neq0$.
It follows as in the proof of \refL{L3} that $\gf$ is continuously
differentiable, and then
\begin{align}\label{ez4}
  \ii\gl\gf'(x)=2\gf(x)-2\intoi \gf(y)\dd y.
\end{align}
Let $\go:=-2/\gl$. Then \eqref{ez4} implies $\gf''(x)=\ii\go\gf'(x)$, and thus
\begin{align}\label{ez5}
  \gf(x)=Ae^{\ii\go x}+B
\end{align}
for some constants $A$ and $B$.
Simple calculus  shows that then \eqref{ez3} holds if and only if
$B=0$ and
\begin{align}\label{ez6}
  e^{\ii\go}=1,
\end{align}
and thus
\begin{align}\label{ez7}
  \go = 2\pi k,\qquad k\in\bbZ.
\end{align}
Consequently, if $\go\neq0$ is as in \eqref{ez7}, then $e^{\ii\go x}$ is an
eigenfunction with eigenvalue $\ii\gl=-2\ii/\go = -\ii/(\pi k)$,
and
the nonzero eigenvalues of $T_{f_{12}}$
are $\set{\frac{\ii}{\pi k}: 0\neq k\in \bbZ}$.
In the notation of \refT{TC} we thus have 
$(\gla_q)_1^\Qp=(\frac{1}{\pi k})_1^\infty$,
and thus \eqref{tc2} yields
\begin{align}\label{ez8}
  \frac{1}{n}\Uc_n(f) \dto \sumk \frac{1}{\pi k}\eta_k,
\end{align}
where $\eta_k$ are \iid{} with the stochastic area distribution
\eqref{area}.
This is equivalent to the limit theorem in \cite[Corollary 2]{writhe}
(there proved by the method of moments)
with the limit represented as in 
\cite[Section 5.3]{writhe}. 
For properties and other descriptions of the limit, see \cite{writhe}.
\end{example}

\begin{example}[alternating inversion number]\label{Einv}
We continue to consider uniformly random permutations in $\fS_n$ 
as in \refR{Rperm} and \refE{Ewrithe}, using the kernel $f$ in \eqref{fsgn}.
Let $X_1,\dots,X_n$ be \iid{} with $X_i\in U\oi$ as in \refR{Rperm}.
Then, as noted above, $U_n(f;\bgs)=U_n(f;X_1,\dots,X_n)$ is, 
up to a trivial linear transformation, the classical inversion number, and 
$\Uc_n(f;\bgs)=\Uc_n(f;X_1,\dots,X_n)$ is the writhe studied in \cite{writhe}.
Furthermore, \cite{writhe} also defines the 
\emph{alternating inversion number} as $\Ump_n(f;\bgs)$
and the
\emph{bi-alternating inversion number} as $\Umm_n(f;\bgs)$.
The alternating inversion number $\Ump_n(f;\bgs)$ was 
(up to the same trivial linear transformation) 
earlier introduced and studied by \cite{Chebikin}, who showed that it has
the same distribution as the ``inversion number'' $U_n(f;\bgs)$.
(This is easily seen by regarding the random permutation $\bgs$
as a random linear order on $[n]$, and constructing it
recursively, by inserting a new element $n$ in a random position
relative to the previous $n-1$ elements; 
the number of new inversions is uniformly random in \set{0,\dots,n-1}, and 
it follows by induction that $U_n(f;\bgs)$ and $\Upm_n(f;\bgs)$ have the
same distribution. And so has $\Ump_n(f;\bgs)$ by \eqref{fi2} since these
variables are symmetric.)

As a consequence,  $\Upm_n(f;\bgs)$ and $\Ump_n(f;\bgs)$ have the same
asymptotic normal distribution as $U_n(f;\bgs)$. From the perspective of our
theorems, this follows because $f$ is antisymmetric, as noted in
\refS{SSummary}. 

In contrast, as shown by \cite{writhe}, 
the bi-alternating inversion number $\Umm_n(f;\bgs)$ has a different limit
distribution, of the degenerate type. As noted above, \cite{writhe} showed
that $\Umm_{2n+1}(f;\bgs)\eqd\Uc_{2n+1}(f;\bgs)$, which generalizes to
\refP{PC=B}; \cite{writhe} also shows 
$\Umm_{2n}(f;\bgs)\eqd\Umm_{2n+1}(f;\bgs)$, which in contrast 
seems to be a very special property for this $f$.
As a consequence, 
the bi-alternating inversion number has the same asymptotic distribution as
the writhe in \eqref{ez8}, which was found by different methods in
\cite{writhe}. 
\end{example}

\section{Further results and open problems}\label{Sfurther}

\subsection{Joint convergence}\label{SSjoint}
We may also consider joint convergence of the different \Ustat{s}
$U_n(f)$, $\Uc_n(f)$, $\Ump_n(f)$, $\Upm_n(f)$, and $\Umm_n(f)$ for the same
kernel $f$.
In the nondegenerate cases, this is straightforward, since the proofs above
in all cases approximate the \Ustat{} by a linear combination of $f_1(X_i)$
and $f_2(X_i)$, and the central limit theorem implies that
these linear combinations converge jointly (after 
normalization by $n^{-3/2}$) to some jointly normal limits.

Also in the degenerate cases, or when some \Ustat{s} are nondegenerate and
some degenerate (and we thus normalize them differently),
the methods above make it in principle possible to study also 
asymptotic joint distributions of $U_n$, $\Ump$, $\Upm$, and $\Umm$,
see \refR{RB} for a simple example; furthermore,
it seems possible to include also $\Uc_n$ by considering 
$\tX_i:=(X_{2i-1},X_{2i},X_{2i-1+n/2},X_{2i+n/2})$
(for $n$ divisible by 4).
We leave such extensions to the reader.

Note that in all results above showing  convergence in distribution to some
limit, 
both in nondegenerate and degenerate cases,
the theorems and proofs also show convergence of first and second moments.
Hence, if $U'_n$ and $U''_n$ denote two of the \Ustat{s} in this paper,
and $\tU'_n$ and $\tU''_n$ are the corresponding normalized variables, 
then the squares $|\tU'_n|^2$  and $|\tU''_n|^2$ are uniformly integrable,
see \eg{} \cite[Theorem 5.5.9]{Gut}. It follows, by the \CSineq, that
also the product $\tU'_n\tU''_n$ is uniformly integrable, and thus
if $\tU_n$ and $\tU''_n$ have limits in distribution jointly, then the
covariance of their limits is the limit of their covariances.
In particular, when the limits are jointly normal, we can easily find their
joint distribution.

In the case when both $U'_n$ and $U''_n$ have limits of the degenerate type,
and we thus may replace $f$ by $f_{12}$ (up to negligible terms), it follows
easily from \refL{LH} that, 
except in the case $(U_n,\Uc_n)$,
the two different \Ustat{s} have covariance of
order $o(n^2)$ because of cancellations caused by the alternating signs.
Hence, except for $(U_n,\Uc_n)$, any joint limits have to be uncorrelated.
(In particular, the two limits cannot be the same random variable, unless
they are 0.) We conjecture that, more strongly, in these cases there is
joint convergence to independent limits, but we leave this as an open problem.
On the other hand, $(U_n,\Uc_n)$ is different:
if we further assume that $f$ is
symmetric, then $U_n(f)-\Uc_n(f)$ is negligible, see \refR{Rsymm},
and thus
$U_n(f)$ and $\Uc_n(f)$ jointly converge, after normalization, to the same
limit. 

\subsection{Strong law of large numbers}\label{SSstrong}
We stated in \eqref{tu0} the  weak law of large numbers for
classical \Ustat{s}. There is also a well-known corresponding
strong law of large numbers,
see
\cite{Hoeffding-LLN} (the symmetric case) and \cite{SJ332} (the general case):
\begin{align}\label{Ulln}
  \frac{1}{\binom n2} U_n(f)\asto \mu.
\end{align}
This extends to the alternating \Ustat{s} in the following, less interesting
form; recall  that by \eqref{tb00} and \eqref{ta00}, the expectations are
$O(n)$. 
\begin{theorem}\label{TSLNN}
  We have, as \ntoo, for any $f\in L^2$,
\begin{align}\label{BAlln}
  \frac{1}{\binom n2} \Umm_n(f)\asto 0.
\end{align}
The same holds for $\Ump_n$ and $\Upm_n$.
\end{theorem}
\begin{proof}
We treat $\Upm_n$; the same argument works for $\Ump_n$ and $\Umm_n$
with minor modifications (and some simplifications).
We use \eqref{Aaida} and treat the terms on the \rhs{} separately.
As just noted, the expectation $\E[\Upm_n(f)]=O(n)$, so the first term in
\eqref{Aaida} is $o(n^2)$. The second term has varianc $O(n)$, and 
it follows from Chebyshev's inequality and the Borel--Cantelli lemma that   
it is $o(n^2)$ a.s. For the third term, let 
\begin{align}
S_k:=\sum_{j=1}^k(-1)^jf_2(X_j),  
\end{align}
and note that by the law of large numbers, applied to even and odd indices
separately,
$S_k/k\asto0$, i.e.,  $S_k=o(k)$ a.s.
Then the third term in \eqref{Aaida} can be written
\begin{align}
  \sumjn\sum_{i=1}^{j-1}(-1)^j f_2(X_j)
=
  \sumin\sum_{j=i+1}^{n}(-1)^j f_2(X_j)
=\sumin(S_n-S_i)=o(n^2)\quad\text{a.s.}
\end{align}
The double sum in \eqref{Aaida} is $\Upm_n(f_{12})$.
For even $n$ we use \eqref{AUU}, where 
the sum again has variance $O(n)$ and thus a.s.\ is $o(n^2)$,
and $U_n(F)/\binom n2\asto0$ by \eqref{Ulln} since
$\E[F(\tX_1,\tX_2)]=0$ by \eqref{AF}.
Finally, for odd $n$, the result follows from \eqref{Aaid3}, where the sum
again has variance $O(n)$.
\end{proof}

Note that this argument does not work for $\Uc_n$, since the definition
\eqref{CtX} involves $n$ explicitly, and we therefore cannot apply \eqref{Ulln}
to $U_n(F)$. Moreover, it is not clear that it is interesting to study
the sequence $(\Uc_n(f))_{n=1}^\infty$ as a stochastic process, since the
point of the definition \eqref{Uc} of $\Uc_n$ is that the indices are
regarded as elements of $\bbZ/n\bbZ$. Nevertheless, out of mathematical
curiosity, we might ask:
\begin{problem}\label{PCLLN}
  Does \eqref{Ulln} hold for $\Uc_n(f)$?
\end{problem}

\subsection{Functional limit theorems}
As another aspect of regarding
the sequence $(U_n(f))_{n=1}^\infty$ and its variants as stochastic
processes, 
we may ask
for functional limit theorems of Donsker-type.
For the classical \Ustat{} $U_n(f)$, it is known that,
extending \eqref{tu1},
$n\qqw\bigpar{U_{\floor{nt}}(f)-\frac{n^2}2t^2\mu}$, regarded as a stochastic
process with continuous parameter $t\ge0$, 
converges in $D\ooo$, as  \ntoo, to a continuous centred Gaussian process;
similarly, in the degenerate case,
$n\qw\bigpar{U_{\floor{nt}}(f)-\E[U_{\floor{nt}}(f)]}$
converges to a continuous process whose
marginals are of the type \eqref{tu2};
see \eg{} \cite{MillerSen}, \cite{Neuhaus1977}, \cite{Hall1979}, 
\cite[Remark 11.11]{SJIII}  
(the symmetric case); 
\cite[p.~83]{SJ22}, 
\cite[Remarks 11.11 and 11.25]{SJIII},  
\cite[Theorem 3.2]{SJ332} 
(the general case).
It seems likely that this too extends to the alternating \Ustat{s} by 
arguing using \eqref{BUU} and \eqref{AUU}, but we have not checked the
details and leave this to the reader.

As in \refSS{SSstrong}, and for the same reason,
this argument does not apply to $\Uc_n$; moreover,
it seems less interesting to consider functional limit theorems for $\Uc_n$.

\subsection{Moment convergence}\label{SSmom}
As noted in \refSS{SSjoint},
in all results above showing  convergence in distribution,
we also have convergence of first and second moments.

For higher moments, it is known that for the classical normal limit in
\eqref{tu1}, and any $p\in(2,\infty)$,
all moments and absolute moments of order $\le p$ converge provided
$\E|f(X_1,X_2)|^p<\infty$, see \cite[Theorem 3.15]{SJ332}.
It seems likely that this extends to the cyclic and alternating \Ustat{s}
considered here, using the methods in the proofs above, but we have not
checked the details and leave this to the reader.

We conjecture that there also is a similar result showing moment convergence
in the degenerate cases, but in this case we are not even aware of a general
result for the classical \Ustat{} and moment convergence in \eqref{tu2}.

\appendix
\section{Proof of \refT{TU}}\label{AA}
We give here a proof of \refT{TU}, which  contains
some known result on the asymptotic distribution of $U$-statistics of the
standard type \eqref{U2}
in the special case of order $m=2$. 
We give a proof for completeness, and because we reuse parts of it for other
proofs; 
we also find it instructive to give complete proofs in the case $m=2$, which
avoids some minor complications for larger $m$.
For previous proofs and for
the general case with arbitrary $m$, see, for example,
\cite{Hoeffding,Gregory1977,RubinVitale1980,DynkinM} for the symmetric case,
and \cite[Chapter 11.1--2]{SJIII} for the general (asymmetric) case.

\begin{proof}[Proof of \refT{TU}]
To prove \refT{TU}, we note first that \eqref{tu00} is immediate from the
definitions \eqref{U2} and \eqref{a1mu}.

We have, by \eqref{U2} and \eqref{a2},
\begin{align}\label{winston}
  U_n(f) &= \sum_{1\le i<j\le n}f(X_i,X_j)
\notag\\&
=\tbinom n2f_\emptyset +\sumin (n-i)f_1(X_i) + \sumjn (j-1)f_2(X_j)
+ \sum_{1\le i<j\le n}f_{12}(X_i,X_j)
\notag\\&
=: \SO_n+\SI_n+\SII_n+\SQ_n.
\end{align}
Here $\SO_n=\tbinom n2\mu=\E U_n(f)$. 
In the sequel, we may replace $f$ by $f-\mu$; this does not affect $f_1$,
$f_2$, or $f_{12}$. (Note also that when $f$ is antisymmetric, 
$\mu=\E f(X_1,X_2)=-\E f(X_2,X_1)=-\mu $
and thus $\mu=0$ so $f-\mu=f$ is still antisymmetric.)
We may thus without loss of generality assume that $\mu=0$, and
hence $\E U_n(f)=0$.

We next study the variances of the sums in \eqref{winston}.
The random vectors $\bigpar{f_1(X_i),f_2(X_i)}$ are \iid, with mean 0 and finite
second moments. Hence,
\begin{align}\label{ab1}
  \Var [\SI_n] &
= \sumin (n-i)^2 \Var[ f_1(\XX)] \sim \tfrac{1}{3}n^3\Var[ f_1(\XX)] ,
\\\label{ab2}
  \Var [\SII_n] &
= \sumjn (j-1)^2 \Var[ f_2(\XX)] \sim \tfrac{1}{3}n^3\Var[ f_2(\XX)] ,
\\\label{ab3}
  \Cov[\SI_n,\SII_n] &
= \sumin(n-i) (i-1) \Cov[f_1(\XX),f_2(\XX)] 
\notag\\&\qquad
\sim \tfrac{1}{6}n^3\Cov[f_1(\XX),f_2(\XX)] .
\end{align}
Consequently, recalling \eqref{gss1},
\begin{align}\label{ab4}
  \Var[\SI_n+\SII_n]&= n^3
\bigpar{\tfrac{1}{3}\Var[ f_1(\XX)] +\tfrac{1}{3}\Var[ f_2(\XX)] 
+\tfrac{2}{6}\Cov[f_1(\XX),f_2(\XX)]+o(1)}
\notag\\&
=n^3\bigpar{\gss+o(1)}.
\end{align}

Turning to $\SQ_n$, we note that the terms $f_{12}(X_i,X_j)$ are identically
distributed and have mean 0, and that they
are orthogonal; this follows from \eqref{a6E}, which implies that
$\E [f_{12}(X_1,X_2)\mid X_1]= \E [f_{12}(X_1,X_2)\mid X_2]=0$, and thus, 
for example,
\begin{align}\label{wba1}
\E\sqpar{f_{12}(X_1,X_2)f_{12}(X_1,X_3)}&
=
\E\bigsqpar{\E[f_{12}(X_1,X_2)\mid X_1]\E [f_{12}(X_1,X_3)\mid X_1]}
\notag\\&
=0.
\end{align}
Consequently,
\begin{align}\label{wba2}
  \Var[\SQ_n]=
\sum_{1\le i<j\le n}\Var[f_{12}(X_i,X_j)]
=
\tbinom n2 \Var[f_{12}(X_1,X_2)].
\end{align}

The variance of $\SQ_n$ is thus $O(n^2)$, while $\SI_n+\SII_n$ typically has
a larger variance of order $n^3$. Hence, the sum \eqref{winston} is
dominated by $\SI_n+\SII_n$, except in the case that $\gss=0$ when these terms
vanish (as we will see below), and therefore \eqref{winston} reduces to $\SQ_n$.
This is the reason for the two different cases \ref{TU1} and \ref{TU2} in
\refT{TU}; the generic  case \ref{TU1}, 
i.e., assuming $\gss>0$, the nondegenerate case,
and the degenerate case \ref{TU2} with $\gss=0$.
We treat these cases separately below, after completing the proof of \ref{TU0}.

\pfitemref{TU0}
We have already shown \eqref{tu00}. Furthermore, \eqref{ab4} and
\eqref{wba2} imply
\begin{align}
  \Var [U_n(f)] = O(n^3),
\end{align}
and thus \eqref{tu0} follows. (Actually, with convergence in $L^2$.)

\pfitemref{TU1}
We apply the standard central limit theorem for triangular arrays 
(see for example \cite[Theorem 7.2.4]{Gut} or \cite[Theorem 5.12]{Kallenberg})
to $\SI_n+\SII_n$.
It is easily verified that the triangular arrays
$\bigpar{n^{-3/2}(n-i)f_1(X_i)}_{i\le n}$ and
$\bigpar{n^{-3/2}(i-1)f_2(X_i)}_{i\le n}$
satisfy the Lindeberg condition, and thus so does the summed array
$\bigpar{n^{-3/2}\xpar{(n-i)f_1(X_i)+(i-1)f_2(X_i)}}_{i\le n}$.
Consequently, the central limit theorem yields, using \eqref{ab4}, 
\begin{align}\label{clt}
n^{-3/2}(\SI_n+\SII_n)\dto N(0,\gss).
\end{align}
(If $\gss=0$ then \eqref{clt} still holds, as a trivial consequence of
\eqref{ab4}.) 
Furthermore, as noted above, \eqref{wba2} implies that 
$\Var[n^{-3/2}\SQ_n]=n^{-3}\Var[\SQ_n]\to0$, and thus
$n^{-3/2}\SQ_n\pto0$. Hence, \eqref{tu1} follows from \eqref{clt} and
\eqref{winston} 
by the Cram\'er--Slutsky theorem \cite[Theorem 5.11.4]{Gut}.

Finally, \eqref{gss1} can be written
\begin{align}
  3\gss = 
\E [(f_1(\XX)+\tfrac12f_2(\XX))^2] + \tfrac34\E [f_2(\XX)^2],
\end{align}
which implies that $\gss=0$ if and only if $f_1(\XX)=f_2(\XX)=0$ a.s.,
which completes the proof for the nondegenerate case \ref{TU1}.

\smallskip
We turn to the degenerate case $\gss=0$ in \ref{TU2}--\ref{TU2a}.
In this case we thus have $f_1(X_i)=f_2(X_i)=0$ a.s., and consequently
$\SI_n=\SII_n=0$ and, by \eqref{winston} again,
\begin{align}
  U_n(f)=\SQ_n.
\end{align}
We first treat the symmetric case \ref{TU2s}.

\pfitemref{TU2s}
We apply \refL{LTsymm} below,
noting that \eqref{lts00} holds by \eqref{a4} and the assumption that
$f_1(\XX)=0$ a.s.\ 
(and our simplifying assumption $f_{\emptyset}=\mu=0$ in this proof). 
This shows that $f(x,y)=f_{12}(x,y)$ has an orthogonal expansion
\eqref{lts1},
for some $R\le\infty$ and some orthonormal sequence of functions
$\gf_r\in\LLR(\cX)$, which furthermore satisfy
\eqref{lts0}, which is equivalent to $\E \gf_r(\XX)=0$.
The orthonormality means that
\begin{align}\label{kv1}
  \E[\gf_r(\XX)\gf_q(\XX)] =\int_{\cX}\gf_r(x)\gf_q(x)\dd\nu(x)=\gd_{rq}.
\end{align}
In other words, $(\gf_r(\XX))_1^R$ is a sequence of uncorrelated random
variables with mean 0 and variance 1.

Suppose first that $R<\infty$, so that the sum \eqref{lts1} is finite. Then
\eqref{U2}, the symmetry of $f$, and \eqref{lts1} yield
\begin{align}\label{kv2}
  2U_n(f)& 
= 2\sum_{1\le i<j\le n} f(X_i,X_j)
=\sum_{i,j=1}^n f (X_i,X_j) -\sumin f(X_i,X_i)
\notag\\&
=\sum_{i,j=1}^n\sumrR \gl_r \gf_r(X_i)\gf_r(X_j)
- \sumin \sumrR\gl_r\gf_r(X_i)^2
\notag\\&
=\sumrR\gl_r\biggpar{\Bigpar{\sumin\gf_r(X_i)}^2-\sumin\gf_r(X_i)^2}.
\end{align}
Now let \ntoo.
By the law of large numbers, for each $r$,
\begin{align}\label{kv3}
n\qw \sumin\gf_r(X_i)^2 \pto \E[\gf_r(\XX)^2]=1.
\end{align}
Furthermore, by the central limit theorem, since $\E[\gf_r(\XX)]=0$ and
$\E[\gf_r(\XX)^2]=1$ as remarked above,
\begin{align}\label{kv4}
n\qqw \sumin\gf_r(X_i) \dto \zeta_r \in N(0,1).
\end{align}
Moreover, since the variables $\gf_r(\XX)$ are uncorrelated, the limit in
\eqref{kv4} holds jointly for all $r\le R$, with the limits $\zeta_r$
uncorrelated and thus independent. Combining \eqref{kv2}--\eqref{kv4} yields
\begin{align}\label{kv5}
  2n\qw U_n(f)& 
=\sumrR\gl_r\biggpar{\Bigpar{n\qqw\sumin\gf_r(X_i)}^2-n\qw\sumin\gf_r(X_i)^2}
\notag\\&
\dto\sumrR\gl_r\bigpar{\zeta_r^2-1}.
\end{align}
This proves \eqref{tu2} when $R<\infty$.

If $R=\infty$, let, for $N\in\bbN$, 
\begin{align}\label{kv6}
  f_N(x,y):=\sum_{r=1}^N \gl_r \gf_r(x)\gf_r(y).
\end{align}
Then the case just proven applies to $f_N$, and thus, for each fixed $N<\infty$,
as \ntoo,
\begin{align}\label{kv7}
  2n\qw U_n(f_N)\dto \sum_{r=1}^N\gl_r\bigpar{\zeta_r^2-1}.
\end{align}
As \Ntoo, the \rhs{} converges to
$\sum_{r=1}^\infty\gl_r\bigpar{\zeta_r^2-1}$
in $L^2$ and a.s., and  in particular in distribution.
Furthermore, by \eqref{wba2} applied to $U_n(f-f_N)$,
\begin{align}\label{kv8}
  \E\bigsqpar{(U_n(f)-U_n(f_N))^2}&
=
  \Var\bigsqpar{U_n(f-f_N)}
=\tbinom n2 \Var\bigpar{(f-f_N)(X_1,X_2)}
\notag\\&
=\tbinom n2 \int_{\cX^2} 
\Bigpar{\sum_{N+1}^\infty\gl_r\gf_r(x)\gf_r(y)}^2\dd\nu(x)\dd\nu(y)
\notag\\&
=\tbinom n2 \sum_{N+1}^\infty \gl_r^2.
\end{align}
Hence, 
\begin{align}\label{kv9}
    \E\bigsqpar{(n\qw U_n(f)-n\qw U_n(f_N))^2}
\le \sum_{N+1}^\infty \gl_r^2
\to0
\end{align}
as \Ntoo, uniformly in $n$. 
The result \eqref{tu2} now follows from \eqref{kv7} and \eqref{kv9},
see \eg{} \cite[Theorem 4.2]{Billingsley}.

Finally, 
by \eqref{tu2} and \eqref{zz1},
\begin{align}
  \Var W = \sumrR\bigpar{\tfrac12\gl_r}^2\Var\bigpar{\zeta_r^2-1}
=\tfrac12\sumrR\gl_r^2
\end{align}
and thus
\eqref{tu22} holds by \eqref{lts2} in \refL{LTsymm}.

\pfitemref{TU2}
Now consider the general degenerate case, where $f_1=f_2=0$, and further
as above without loss of generality $f_\emptyset=\mu=0$, but no symmetry
assumption is made.
We use the following trick to reduce to the symmetric case.
(See \cite[Remark 11.21]{SJIII} for the case of general order $m$.)

Let $(Z_i)\xoo$ be an \iid{} sequence of random variables, 
independent of $(X_i)\xoo$, with each $Z_i$
uniformly distributed on $\oi$.
Consider the random variables $\hX_i:=(X_i,Z_i)$ in $\hcX:=\cX\times\oi$
and define the function $\hf:\hcX^2\to\bbR$ by \eqref{hff}.
Note that this definition makes $\hf$ a symmetric function on
$\hcX\times\hcX$.
Furthermore, if we condition on the sequence $(Z_i)_1^n$, and assume as we
may that $Z_1,\dots,Z_n$ are distinct, then,
letting $\pi$ be the permutation of \set{1,\dots,n} that makes
$Z_{\pi(1)}<\dots<Z_{\pi(n)}$,
\begin{align}
  U_n(\hf;\hX_1,\dots,\hX_n) 
= U_n(\hf;\hX_{\pi(1)},\dots,\hX_{\pi(n)})
= U_n(f;X_{\pi(1)},\dots,X_{\pi(n)}),
\end{align}
where the first equality holds by the symmetry of $\hf$
and the second by the definitions of $\hf$ (in \eqref{hff}) and $\pi$.
Consequently, conditioned on $(Z_i)_1^n$ we have
\begin{align}\label{eleo}
  U_n(\hf;\hX_1,\dots,\hX_n) \eqd U_n(f;X_{1},\dots,X_{n}),
\end{align}
and hence \eqref{eleo} holds also unconditionally.
The result now follows from \ref{TU2s} applied to $\hf$ and $(\hX_i)\xoo$;
note that this case applies since the definition \eqref{hff} implies that,
with definitions analogous to \eqref{a3}--\eqref{a6},
\begin{align}
\hf_\emptyset&=
\int_{\oi^2}\int_{X^2}\hf\bigpar{(x,t),(y,u)}\dd\nu(x)\dd\nu(y)\dd t\dd u
=\mu=0,
\\
  \hf_1(x,t)&=\int_{\hcX}\hf\bigpar{(x,t),(y,u)}\dd\nu(y)\dd u
=
\begin{cases}
f_1(x), & t<u
\\
f_2(x), & t>u\end{cases}
\biggr\}
=0
\end{align}
and, by symmetry, $\hf_2(x,t)=\hf_1(x,t)=0$.
Furthermore, 
\eqref{tu22} for $f$ follows from the same formula for $\hf$,
since
\eqref{ff} shows  that
$\Var[\hf(\hX_1,\hX_2)] =\Var[f(X_1,X_2)]$.

\pfitemref{TU2a}
This is proved directly in \cite[Theorem 2.1]{SJ22} by different methods,
relating $U_n$ in the asymmetric case to the stochastic area process.
We give here a different proof, by combining the general result
in \ref{TU2} with \refL{Lanti} below, which 
finds the eigenvalues $\gl_r$ of $T_\hf$ and
shows 
that the limit variable 
$ W:=\sumrR\tfrac12\gl_r(\zeta_r^2-1)$ in \eqref{tu2} also has the
representation \eqref{tu2a}. The formula \eqref{tu2a2} follows from
\eqref{tu22} and \eqref{lanti3}.
\end{proof}

\begin{lemma}\label{LTsymm}
If\/ $f\in L^2(\cX\times\cX)$ is real and symmetric, then
$T_f$ defined by \eqref{Tg}, i.e.,
\begin{align}\label{tf1}
  T_f g(x) := \int_{\cX}f(x,y)g(y)\dd\nu(y),
\end{align}
is a self-adjoint Hilbert--Schmidt operator
on $\LLC(\cX)$. 
Hence $T_f$ is compact and has thus at most countably many nonzero
eigenvalues, each of them real and each having a finite-dimensional eigenspace.
Let $(\gl_r)_{r=1}^R$ (where $0\le R\le\infty$) be an
enumeration of the nonzero eigenvalues (with  multiplicities) of $T_f$.
It is then possible to find 
a corresponding orthonormal sequence of real-valued eigenfunctions 
$(\gf_r)_1^R\in\LLR(\cX)$
such that $T_f\gf_r=\gl_r\gf_r$ for every $r$.

For any such $(\gl_r)_1^R$ and $(\gf_r)_1^R$,
$f(x,y)$ has a (finite or infinite) orthogonal expansion
\begin{align}\label{lts1}
  f(x,y) = \sumrR \gl_r\gf_r(x)\gf_r(y)
\end{align}
which converges in $L^2(\cX^2)$ because
\begin{align}\label{lts2}
\sumrR\gl_r^2 
= \int_{\cX^2}f(x,y)^2\dd\nu(x)\dd\nu(y)
<\infty
.\end{align}

Moreover, if
\begin{align}\label{lts00}
  \int_{\cX} f(x,y)\dd\nu(x)=0,
\qquad \text{for $\nu$-a.e. } y\in\cX,
\end{align}
then
\begin{align}\label{lts0}
  \int_{\cX} \gf_r(x)\dd\nu(x)=0,
\qquad \text{for every } r\le R.
\end{align}
\end{lemma}
Here and below, $r\le R$ should be interpreted as $r<\infty$ when $R=\infty$.
\begin{proof}
  Since $f\in L^2(\cX^2)$, \eqref{tf1} defines a bounded and compact linear
  operator $T_f$ on $\LLC(\cX)$; furthermore, $T_f$ is a Hilbert-Schmidt
  operator. 
Since furthermore $f$ is real and symmetric, $T_f$ is a
self-adjoint operator.
By the spectral theorem for compact and self-adjoint linear operators on a
Hilbert space
(see \eg{} \cite[Theorem 28.3]{Lax} or \cite[Theorem 6.4-B]{Taylor}),
$T_f$ has a finite or countably infinite set of
nonzero eigenvalues, each with finite multiplicity, so we may
arrange the nonzero
eigenvalues, with multiplicities, in a sequence $(\gl_r)_1^R$ with $R\le\infty$;
we denote the index set also by $\cR:=\set{r\in\bbN:r\le R}$.
Moreover, the eigenvalues $\gl_r$ are real, and 
there exists a corresponding orthonormal sequence of
eigenfunctions $(\gf_r)_1^R$, and this may be extended to an orthonormal
basis $(\gf_r)_{r\in\cR\cup\cN}$ where 
$\cN$ is a disjoint (possibly empty) index set
such that $T_f(\gf_r)=0$ for every $r\in\cN$, i.e., each $\gf_r$ is an
eigenfunction also for $r\in\cN$, with eigenvalue $\gl_r:=0$ when $r\in\cN$.

Furthermore, since $f$ is real, 
we may choose all $\gf_r$ to be in $\LLR(\cX)$.
(By \cite[Theorem 6.4-B]{Taylor} applied to $\LLR(\cX)$,
or by noting that 
since also the eigenvalues are real,
the real and imaginary parts of any eigenfunction of $T_f$ are also
eigenfunctions for the same eigenvalue; hence each eigenspace is spanned
by the real functions in it.)

Fix any such sequence $(\gf_r)_{r\in\cR}$ and extension $(\gf_r)_{r\in\cR\cup\cN}$.
Since $(\gf_r)_{r\in\cR\cup\cN}$ is an orthonormal basis in $L^2(\cX)$,
it is easily seen (and well-known) that 
if we define $g\tensor h(x,y):=g(x)h(y)$ for functions $g,h\in L^2(\cX)$, then
the set
$\set{\gf_r\tensor\gf_s:r,s\in\cR\cup\cN}$ is an orthonormal basis in
$L^2(\cX\times\cX)$. Hence,
\begin{align}\label{wt}
  f = \sum_{r,s\in\cR\cup\cN}\innprod{f,\gf_r\tensor\gf_s}\,{\gf_r\tensor\gf_s},
\end{align}
where the sum converges in $L^2$.
By Fubini's theorem and \eqref{tf1}, 
\begin{align}\label{wt1}
  \innprod{f,\gf_r\tensor\gf_s}&
=\iint_{\cX\times\cX}f(x,y)\gf_r(x)\gf_s(y)\dd\nu(x)\dd\nu(y)
=\int_{\cX}\gf_r(x)T_f(\gf_s)(x)\dd\nu(x)
\notag\\&
=\int_{\cX}\gf_r(x)\gl_s\gf_s(x)\dd\nu(x)
=\gl_s\innprod{\gf_r,\gf_s}
=\gl_s\gd_{rs}.
\end{align}
This vanishes unless $r=s\in\cR$, and thus \eqref{wt} simplifies to,
using \eqref{wt1} again, 
\begin{align}\label{wt2}
    f = \sum_{r\in\cR}\innprod{f,\gf_r\tensor\gf_r}\,{\gf_r\tensor\gf_r}
= \sum_{r\in\cR}\gl_r \,\gf_r\tensor\gf_r.
\end{align}
This is \eqref{lts1}, and \eqref{lts2} follows because
$\set{\gf_r\tensor\gf_s}$ is an orthonormal basis.

Finally, if \eqref{lts00} holds, then for every $r\in\cR$,
by \eqref{tf1} and Fubini's theorem,
\begin{align}
  \gl_r\int_{\cX}\gf_r(x)\dd\nu(x)&
=
  \int_{\cX}T_f(\gf_r)(x)\dd\nu(x)
=
  \iint_{\cX\times\cX}f(x,y)\gf_r(y)\dd\nu(y)\dd\nu(x)
\notag\\&
=\int_{\cX} \gf_y(y) \int_{\cX} f(x,y)\dd\nu(x) \dd\nu(y)=0.
\end{align}
Since $\gl_r\neq0$ for $r\in\cR$, \eqref{lts0} follows.
\end{proof}

\begin{lemma}\label{Lanti}
  Let $f\in L^2(\cX\times\cX)$ be real and antisymmetric.
The the operator $T_f$ on $\LLC(\cX)$
is anti-self-adjoint and has purely imaginary eigenvalues.
Let $(\gll_q)_{q\in\cQp}$ be an enumeration of the positive real numbers
such that\/ 
$\ii\gll_q$ is an eigenvalue of $T_f$ (counted with multiplicities).
Then the multiset of nonzero eigenvalues $(\gl_r)_1^R$
of the self-adjoint operator $T_{\hf}$ on
$L^2(\cX\times\oi)$ (counted with multiplicities)
 equals
\begin{align}\label{lanti1}
\Bigset{\pm\frac{2}{(2k-1)\pi}\gll_q:q\in\cQp,k\in\bbN}\quad
\text{with each  pair $(q,k)$ counted twice.}
\end{align}
As a consequence, if $\zeta_r$ are \iid{} standard normal variables and
$\eta_q$ are independent random variables with the stochastic area
distribution \eqref{area}, then
\begin{align}\label{lanti}
  \sumrR \tfrac12\gl_r(\zeta_r^2-1) \eqd \sum_{q\in\cQp} \gll_q\eta_q.
\end{align}
Furthermore,
\begin{align}\label{lanti3}
  \sumrR\gl_r^2
=2\sumqQp(\gll_q)^2
= \int_{\cX^2}f(x,y)^2\dd\nu(x)\dd\nu(y)
<\infty
.\end{align}
\end{lemma}
\begin{proof}
In the antisymmetric case, 
we  can write \eqref{hff} as
\begin{align}\label{er1}
  \hf\bigpar{(x,t),(y,u)} = f(x,y)\sgn(u-t)
=:f(x,y)h(t,u),
\end{align}
where $\sgn$ is the sign function \eqref{sgn}.
Thus, in tensor notation, see \eqref{tensor},
$\hf=f\tensor h$ and $T_\hf=T_f\tensor T_h$.

The functions $f(x,y)$ and $h(t,u)=\sgn(u-t)$ in \eqref{er1} 
are both real-valued and
antisymmetric, and thus the corresponding 
Hilbert--Schmidt
integral operators $T_f$ and $T_h$
(acting on $\LLC(\cX)$ and $\LLC\oi$, respectively)  are both anti-self-adjoint.
Hence, $-\ii T_f$ and $-\ii T_h$ are self-adjoint, and it follows from the
spectral theorem, as in the proof of \refL{LTsymm},
that $-\ii T_f$ and
$-\ii T_h$ have only real eigenvalues $\set{\gll_q:q\in\cQ\cup\cN}$ and
$\set{\rho_s:s\in\cS\cup\cN'}$, respectively,
with $\gll_q\neq0\iff q\in\cQ$
and $\rho_s\neq0\iff s\in\cS$
and that there are corresponding families
of eigenfunctions $\set{\gf_q:q\in\cQ\cup\cN}$ and
$\set{\psi_s:s\in\cS\cup\cN'}$ which are orthonormal bases in $\LLC(\cX)$
and $\LLC\oi$, respectively. (However, unlike in \refL{LTsymm},
these eigenfunctions are not real-valued.)
Hence, these functions are eigenfunctions for $T_f$ and $T_h$ too, with
eigenvalues $\ii\gll_q$ and $\ii\rho_s$, respectively.
(The eigenvalues and eigenfunctions for $T_h$ will be found explicitly in
\refL{L3}.)

It follows that the set of all functions
$\gf_q\tensor\psi_s(x,t):=\gf_q(x)\psi_s(t)$ is an 
orthonormal basis in $\LLC(\cX\times\oi)$.
Furthermore,  
as noted in \refS{Sprel},
the function $\gf_q\tensor\psi_s$ 
is an eigenfunction of $T_\hf=T_f\tensor T_h$
with eigenvalue $-\gll_q\rho_s$. Since these functions form a basis, it
follows that the set of eigenvalues of $T_\hf$, with multiplicities, is
$\set{-\gll_q\rho_s:q\in\cQ\cup\cN, s\in\cS\cup\cN'}$.
In particular, the nonzero eigenvalues $(\gl_r)_1^R$ are
\begin{align}\label{er3}
  \set{-\gll_q\rho_s:q\in\cQ, s\in\cS}.
\end{align}

Recall that the nonzero
eigenvalues of $T_f$ are $\set{\ii\gll_q}_{q\in\cQ}$, where $\gll_q\in\bbR$. 
Since $f$ is real, the complex conjugate $\overline{\gf_q}$ is also an
eigenfunction, with eigenvalue $\overline{\ii\gll_q}=-\ii\gll_q$.
It follows that if we let $\cQ_+:=\set{q:\gll_q>0}$
and $\cQ_-:=\set{q:\gll_q<0}$, then 
$\set{\gll_q:q\in\cQ_-}=\set{-\gll_q:q\in\cQ_+}$.
Consequently, we may rewrite \eqref{er3} as
\begin{align}\label{er3b}
  \set{\gl_r:r\le R} = 
  \set{\pm\gll_q\rho_s:q\in\cQp, s\in\cS}.
\end{align}

We now use \refL{L3}, which shows that the eigenvalues $\ii\rho_s$ are 
\begin{align}\label{er3c}
\Bigset{\pm\frac{2\ii}{(2k-1)\pi}: k\in\bbN}. 
\end{align}
Hence, \eqref{lanti1} follows
from \eqref{er3b}, noting that for each pair $(q,k)$, there are two choices
of signs in \eqref{er3b} and \eqref{er3c} that yield the same $\gl_r$.

Note that each pair $(q,k)$ thus yields 4 eigenvalues in \eqref{lanti1},
2 of each sign.
Hence, it follows from \eqref{lanti1} that, 
with $\zeta_{q,k,j}\in N(0,1)$ independent,
\begin{align}\label{lanti=}
  \sumrR \tfrac12\gl_r(\zeta_r^2-1) 
\eqd\sum_{q\in\cQp}\gll_q \sumk \frac{1}{(2k-1)\pi} 
\bigpar{\zeta^2_{q,k,1}+\zeta^2_{q,k,2}-\zeta^2_{q,k,3}-\zeta^2_{q,k,4}}.
\end{align}
Consequently, \eqref{lanti} follows from \refL{Larea}.
Finally, \eqref{lanti3} follows from \eqref{lanti1} (or from
\eqref{er3b}--\eqref{er3c}) which yields
\begin{align}
  \sumrR\gl_r^2
=
4\sum_{q\in\cQp}\sumk\Bigpar{\frac{2}{(2k-1)\pi}\gll_q}^2
=
\frac{16}{\pi^2}\sumk\frac{1}{(2k-1)^2}\sum_{q\in\cQp}(\gll_q)^2
=2\sum_{q\in\cQp}(\gll_q)^2
,\end{align}
together with \eqref{lts2} for $T_\hf$
(or the corresponding formula for the self-adjoint operator $T_{\ii f}$).
\end{proof}

\begin{lemma}\label{L3}
Let $h(t,u):=\sgn(t-u)$. Then the anti-self-adjoint operator 
$T_h$ acting on 
$\LLC\oi$ has
eigenvalues, all simple,
\begin{align}\label{l3}
{\pm \frac{2\ii}{(2k-1)\pi}, \qquad k=1,2,3,\dots}  
.\end{align}
\end{lemma}
\begin{proof}
  Suppose that $\gf$ is an eigenfunction with eigenvalue $\gl$.
Then, for a.e.\ $t\in\oi$,
  \begin{align}\label{sw1}
    \gl \gf(t) = T_h\gf(t) = \intoi \sgn(u-t)\gf(u)\dd u
=-\int_0^t\gf(u)\dd u +\int_t^1\gf(u)\dd u
  \end{align}
Suppose first that $\gl\neq0$.
The \rhs{} of \eqref{sw1} is a continuous function of $t\in\oi$, and thus $\gf$
can be assumed to be continuous. 
Then \eqref{sw1} holds for every $t\in\oi$, and 
the \rhs{} of \eqref{sw1} is continuously differentiable in $(0,1)$;
thus $\gf$ is continuously differentiable on $\oi$.
Taking the derivative in \eqref{sw1} yields
\begin{align}
\gl\gf'(t)=-2\gf(t)
\end{align}
and thus (for some irrelevant $C\neq0$)
\begin{align}\label{sw3}
  \gf(t)=Ce^{i\go t}
\qquad\text{with}\qquad 
\go=2\ii/\gl.
\end{align}
Taking $t=0$ and $1$ in \eqref{sw1} yields
\begin{align}
  \gl\gf(0)=\intoi\gf(u)\dd u=-\gl\gf(1),
\end{align}
i.e., $\gf(1)=-\gf(0)$. This and \eqref{sw3} yield
$e^{\ii\go}=-1$, and thus
\begin{align}\label{sw5}
  \go=\pm(2k-1)\pi,
\qquad k=1,2,\dots.
\end{align}
Conversely, it is easily checked that each such $\go$ gives an eigenfunction
$\gf$ by \eqref{sw3}, satisfying \eqref{sw1} with eigenvalue
\begin{align}
  \gl=\frac{2\ii}{\go}=\pm\frac{2\ii}{(2k-1)\pi}.
\end{align}
These are thus the nonzero eigenvalues, and we see from \eqref{sw3} that
they are simple.

Finally, if \eqref{sw1}  holds with $\gl=0$, then $\int_0^t\gf(u)\dd u$ is
constant, and thus $\gf(t)=0$ a.e. Hence, $0$ is not an eigenvalue.
(Equivalently, the eigenfunctions $e^{\pm(2k-1)\pi\ii}$ form an orthonormal
basis on $\oi$, as is well-known from Fourier analysis.)
\end{proof}

\section{Cumulants}\label{Acumulants}
The limit distributions in our theorems are, apart form the normal
distribution, given by (possibly infinite) linear combinations of
independent copies of the variables $\zeta^2-1$, $\eta$, and $\xeta$ in
\refSS{SS3}. We give here some simple results on the cumulants of such sums;
we denote the cumulants of a random variable $Y$ by $\gk_m(Y)$, where
$m\ge1$.
We write
for convenience $\chi:=\zeta^2-1$.

The cumulants $\gk_m(\chi)$, $\gk_m(\eta)$, and $\gk_m(\xeta)$ 
are (by definition) obtained by
Taylor expansions of the logarithms of the
\chf{s} of $\chi$, $\eta$, and $\xeta$ given in \eqref{zz2}
\eqref{area}, and \eqref{xeta};
note that these \chf{s} and their logarithms are analytic functions of $t$ in a
neighbourhood of $0$. This yields
for $m\ge2$ the cumulants, using \cite[4.19.8, 4.28.9, and 24.2.2]{NIST}
for \eqref{kk2}--\eqref{kk3},
\begin{align}\label{kk1}
  \gk_m(\chi) &=2^{m-1}(m-1)!,
\\\label{kk2}
\gk_m(\eta)&=(-1)^{1+m/2}\frac{2^m(2^m-1)}{m}B_m
=\frac{2^m(2^m-1)}{m}|B_m|,
\\\label{kk3}
\gk_m(\xeta)&=\tfrac12\gk_m(\eta)=(-1)^{1+m/2}\frac{2^{m-1}(2^m-1)}{m}B_m
=\frac{2^{m-1}(2^m-1)}{m}|B_m|,
\end{align}
where $B_m$ denotes the Bernoulli numbers \cite[Chapter 24]{NIST}.
For $m=1$, we have $\gk(\chi)_1=\gk(\eta)_1=\gk(\xeta)_1=0$. (These are
just the means.) Note that thus $\gk_m(\eta)=\gk_m(\xeta)=0$ for all odd $m$,
which reflects the fact that $\eta$ and $\xeta$ have symmetric
distributions.

Sums of the type
$\sumrR\gl_r\chi_r$, $\sumrR\gl_r\eta_r$, and $\sumrR\gl_r\xeta_r$
(where $\chi_r$ are independent copies of $\chi$, and so on)
 appear frequently above; they
have cumulants that can be expressed in terms of the sums
$\sumrR\gl_r^m$,
since the cumulant of a sum of independent variables equals the sum of their
cumulants. 
Hence, for any finite or infinite sequence $(\gl_r)_1^R$ with
$\sumrR\gl_r^2<\infty$, 
\begin{align}
  \label{kk4}
\gk_m\Bigpar{\sumrR\gl_r\chi_r}
=\sumrR\gl_r^m\cdot\gk_m(\chi),
\end{align}
with $\gk_m(\chi)$ given by \eqref{kk1}, and similarly for 
$\sumrR\gl_r\eta_r$ and $\sumrR\gl_r\xeta_r$.

For example, since $\eta\eqd\xeta_1+\xeta_2$, we have
$\gk_m(\eta)=2\gk_m(\xeta)$ for all $m\ge1$, as is seen in
\eqref{kk2}--\eqref{kk3}.
Similarly, since \eqref{larea0} can be written 
\begin{align}
  \xeta\eqd
\sum_{k=-\infty}^\infty\frac{1}{(2k-1)\pi} \chi_k,
\end{align}
we have, for all even $m\ge2$, 
using the standard formula \cite[25.6.2]{NIST} for
$\zeta(m)$,
\begin{align}
\gk_m(\xeta)&=
\sum_{k=-\infty}^\infty\frac{1}{((2k-1)\pi)^m}\gk_m(\chi)
=\pi^{-m}2\bigpar{1-2^{-m}}\zeta(m)\gk_m(\chi)
\notag\\&
=\frac{2^{m}\bigpar{1-2^{-m}}}{m!}|B_{m}|\gk_m(\chi)
=\frac{2^{m}\bigpar{1-2^{-m}}}{m!}|B_{m}|2^{m-1}(m-1)!,
\end{align}
which agrees with (and thus gives a proof of) \eqref{kk3} and \eqref{kk2}.
(Recall that the odd cumulants of $\zeta$ and $\eta$ are 0.)

For a final example, the limit distribution $\sumk\frac{1}{\pi k}\eta_k$
in \eqref{ez8} has cumulants, using \eqref{kk2} and, again, \cite[25.6.2]{NIST},
\begin{align}
\gk_m\Bigpar{\sumk\frac{1}{\pi k}\eta_k}
&=
\sumk\frac{1}{(\pi k)^m}\gk_m(\eta)
=\pi^{-m}\zeta(m)\gk_m(\eta)
=\frac{2^{2m}(2^m-1)}{2m\cdot m!}B_{m}^2,
\end{align}
which agrees with the cumulants given (implicitly) in 
\cite[Corollary 2]{writhe}. (Note that the variable $W$ there is twice as
big, because of different normalizations.)


\section{A proof of \eqref{qj1}--\eqref{york}}\label{AYor}
We give here,
as another illustration of the theorems and methods in the paper, 
a proof of \eqref{york} and thus \eqref{yor} using \refT{TU}
and eigenvalue calculations; this is hardly new, but we do not know a
reference. We omit some details.

As said in \refR{RWP}, 
we leave it as an open problem to do similar calculations for
the operator \eqref{ta21} in \refT{TA}, which ought to lead to 
an explicit (more or less complicated) formula for the joint \chf{} of the
limits in distribution of $n\qw\Upm_n(\fs)$ and $n\qw\Upm_n(\fa)$, 

Let $\fs$ and $\fa$ be as in \eqref{ewis}--\eqref{ewia}, and let
$s,\tau\in\bbR$.
(We use $\tau$ here, since we want to use $t\in I$ as one of the
coordinates in $\cX\times\oi$.)
Take, suppressing the argument 
$\xpar{(x_1,x_2),(x_1',x_2')}$,
\begin{align}\label{wp1}
  f\st
:=s\cdot2\fs+\tau\cdot2\fa = ax_1x_2'+bx_2x_1',
\end{align}
where $a:=s+\tau$, $b:=s-\tau$.
Similarly as in \eqref{ewi5}--\eqref{ewi7}, we see that an eigenfunction of
$T_\hfst$ with nonzero eigenvalue $\gl$ has to be of the form
$x_1\psi_1(t)+x_2\psi_2(t)$, and the eigenvalue equation is equvalent to the
system
\begin{align}\label{wp21}
  \gl\psi_1(t)&=a\int_t^1\psi_2(u)\dd u  + b\int_0^t\psi_2(u)\dd u, 
\\\label{wp22}
  \gl\psi_2(t)&=b\int_t^1\psi_1(u)\dd u  + a\int_0^t\psi_1(u)\dd u. 
\end{align}
This is in turn equivalent to the system of differential equations
\begin{align}\label{wp31}
  \gl\psi'_1(t)&=(b-a)\psi_2(t),
\\\label{wp32}
  \gl\psi'_2(t)&=(a-b)\psi_1(t),
\end{align}
with the initial values 
\begin{align}\label{wp33}
  \gl\psi_1(0)=a\intoi\psi_2(u)\dd u,
\qquad
  \gl\psi_2(0)=b\intoi\psi_1(u)\dd u
.\end{align}
Assume $s\tau\neq0$, or equivalently $a\neq\pm b$. (This excludes the cases
$\fs$ and $\fa$ already studied, which are somewhat special.)
Let
\begin{align}\label{wp4}
\go= \frac{a-b}{\gl}. 
\end{align}
The general solution to \eqref{wp31}--\eqref{wp32} then is,
for some real (or complex)  $A$ and $B$,
\begin{align}\label{wp51}
  \psi_1(t)&=A\cos(\go x)+B\sin(\go x),
\\\label{wp52}
  \psi_2(t)&=A\sin(\go x)-B\cos(\go x),
\end{align}
and \eqref{wp33} yields, 
using \eqref{wp4} and \eqref{wp31}--\eqref{wp32}, the conditions
\begin{align}\label{wp61}
  (a-b)\psi_1(0)&=\go\gl\psi_1(0)=a\go\intoi\psi_2(u)\dd u
=a\psi_1(0)-a\psi_1(1),
\\\label{wp62}
  (a-b)\psi_2(0)&=\go\gl\psi_2(0)=b\go\intoi\psi_1(u)\dd u
=b\psi_2(1)-b\psi_2(0),
\end{align}
which simplify to
\begin{align}\label{wp7}
  b\psi_1(0)=a\psi_1(1),
\qquad
a\psi_2(0)=b\psi_2(1)
\end{align}
or
\begin{align}\label{wp81}
  bA&=a(A\cos\go+B\sin\go),
\\\label{wp82}
-aB&=b(A\sin\go-B\cos\go).
\end{align}
Regarding \eqref{wp81}--\eqref{wp82} as a system of linear equations in
$(A,B)$, it follows that there is a solution to \eqref{wp21}--\eqref{wp22},
and thus an eigenvalue $\gl$, if and only if the 
determinant of the system \eqref{wp81}--\eqref{wp82} is 0, i.e., if
\begin{align}\label{wp9}
0= (a\cos\go-b)(a-b\cos\go)-ab\sin^2\go=(a^2+b^2)\cos\go-2ab.
\end{align}
Furthermore, since we assume $a\neq\pm b$, it is easily seen that then
the system \eqref{wp81}--\eqref{wp82} has rank 1 and thus a one-dimensional
space of solutions $(A,B)$; hence, the eigenvalue $\gl$ is simple.
Let 
\begin{align}\label{goo}
  \goo:=\arccos\frac{2ab}{a^2+b^2}.
\end{align}
Then the complete set of solutions $\go$ to \eqref{wp9} is 
$\set{\pm\goo+2k\pi}$,
$k\in\bbZ$, and hence, by \eqref{wp4}, 
the nonzero eigenvalues of $T_\hfst$ are (all simple)
\begin{align}
  \label{wp10}
 \pm\frac{a-b}{\goo+2k\pi}
= \pm\frac{2\tau}{\goo+2k\pi}, 
\qquad k\in\bbZ.
\end{align}
Consequently, \eqref{tu2} and \eqref{zz2} yield
\begin{align}\label{wp91}
  n\qw U_n(s\cdot2\fs+\tau\cdot2\fa)
=n\qw U_n(f\st) \dto W\st
\end{align}
where,
using also the product expansion for cosine \cite[4.22.2]{NIST},
\begin{align}\label{wp11}
  \E[e^{\ii W\st}]&
=\prod_{\gl}e^{-\ii\gl/2}(1-\ii\gl)\qqw
=\prod_{k=-\infty}^\infty
\Bigpar{
\Bigpar{1-\frac{2\ii \tau}{\goo+2\pi k}}
\Bigpar{1+\frac{2\ii \tau}{\goo+2\pi k}}
}\qqw
\notag\\&
=\prod_{k=-\infty}^\infty
\Bigabs{1+\frac{2\ii \tau}{\goo+2\pi k}}
\qw
=\prod_{k=-\infty}^\infty
\lrabs{
\frac{1+\frac{\goo+\pi}{(2k-1)\pi}}
{1+\frac{\goo+\pi+2\ii\tau}{(2k-1)\pi}}}
\notag\\&
=\prod_{k=1}^\infty
\lrabs{
\frac{1-\frac{(\goo+\pi)^2}{(2k-1)^2\pi^2}}
{1-\frac{(\goo+\pi+2\ii\tau)^2}{(2k-1)^2\pi^2}}}
=\Bigabs{\frac{\cos(\goo/2+\pi/2)}{\cos(\goo/2+\pi/2+\ii\tau)}}
=\Bigabs{\frac{\sin(\goo/2)}{\sin(\goo/2+\ii\tau)}}
\notag\\&
=\Bigabs{\frac{\sin(\goo/2)}
  {\sin(\goo/2)\cosh(\tau)+\ii \cos(\goo/2)\sinh(\tau)}}
\notag\\&
=\bigpar{\cosh^2(\tau)+\cot^2(\goo/2)\sinh^2(\tau)}
\qqw
.\end{align}
Furthermore, by \eqref{goo},
\begin{align}\label{wp12}
  \cot^2(\goo/2)=\frac{1+\cos(\goo)}{1-\cos(\goo)}
=\frac{(a+b)^2}{(a-b)^2}=\frac{s^2}{\tau^2}
.\end{align}
It follows from \eqref{wp91}--\eqref{wp12} that \eqref{qj1} and \eqref{qj2}
hold jointly, with limits $\Ws$ and $\Wa$ having the joint \chf{}
\eqref{york}.
(The cases $s=0$ or $\tau=0$, implicit in \eqref{qj2} and \eqref{qj1}, 
follow by continuity.)

\newcommand\AAP{\emph{Adv. Appl. Probab.} }
\newcommand\JAP{\emph{J. Appl. Probab.} }
\newcommand\JAMS{\emph{J. \AMS} }
\newcommand\MAMS{\emph{Memoirs \AMS} }
\newcommand\PAMS{\emph{Proc. \AMS} }
\newcommand\TAMS{\emph{Trans. \AMS} }
\newcommand\AnnMS{\emph{Ann. Math. Statist.} }
\newcommand\AnnPr{\emph{Ann. Probab.} }
\newcommand\CPC{\emph{Combin. Probab. Comput.} }
\newcommand\JMAA{\emph{J. Math. Anal. Appl.} }
\newcommand\RSA{\emph{Random Structures Algorithms} }
\newcommand\DMTCS{\jour{Discr. Math. Theor. Comput. Sci.} }

\newcommand\AMS{Amer. Math. Soc.}
\newcommand\Springer{Springer-Verlag}
\newcommand\Wiley{Wiley}

\newcommand\vol{\textbf}
\newcommand\jour{\emph}
\newcommand\book{\emph}
\newcommand\inbook{\emph}
\def\no#1#2,{\unskip#2, no. #1,} 
\newcommand\toappear{\unskip, to appear}

\newcommand\arxiv[1]{\texttt{arXiv}:#1}
\newcommand\arXiv{\arxiv}

\newcommand\xand{and }
\renewcommand\xand{\& }

\def\nobibitem#1\par{}

\end{document}